\documentclass[reqno]{amsart}
\usepackage{aligned-overset}
\usepackage[english]{babel}
\usepackage[utf8]{inputenc}
\usepackage{amssymb}
\usepackage{tikz,tikzscale}
\usetikzlibrary{calc}
\tikzset{>=latex}
\usepackage{enumerate,enumitem}
\usepackage{booktabs,multirow}
\usepackage{amsfonts}
\usepackage{dsfont}
\usepackage{array}
\usepackage{soul}
\usepackage[group-separator={,}]{siunitx}

\usepackage[normalem]{ulem}
\newcounter{corr}
\definecolor{violet}{rgb}{0.580,0.,0.827}
\newcommand{\corr}[3]{\typeout{Warning : a correction remains in page \thepage}
  \stepcounter{corr}        
 	            {\color{blue}\ifmmode\text{\,\sout{\ensuremath{#1}}\,}\else\sout{#1}\fi}
              {\color{red}#2}
              {\color{violet} #3}
}

\usepackage{cite} % 
\usepackage[colorlinks = true,
            linkcolor = blue,
            urlcolor  = blue,
            citecolor = blue,
            anchorcolor = blue]{hyperref} % for hyperlinks in citations
\usepackage[capitalise,nameinlink,noabbrev]{cleveref} % to include "Theorem " (or something else) in the link
\hypersetup{final}
\newtheorem{theorem}{Theorem}[section]
\newtheorem{lemma}[theorem]{Lemma}

\theoremstyle{definition}

\theoremstyle{remark}
\newtheorem{remark}[theorem]{Remark}

\numberwithin{equation}{section}

\newcommand{\elem}{\ensuremath{T}}
\newcommand{\mesh}{\ensuremath{\mathcal T}}

\newcommand{\faceSet}{\ensuremath{\mathcal F}}

\newcommand{\face}{\ensuremath{F}}
\newcommand{\skeleton}{\ensuremath{\faceSet}}
\newcommand{\skeletalSpace}{\ensuremath{M}}

\newcommand{\linElementSpace}{\ensuremath{\overline V^\textup c}}
\newcommand{\discElementSpace}{\ensuremath{V}}

\newcommand{\polynomials}{\ensuremath{\mathcal P}}
\newcommand{\level}{\ensuremath{\ell}}
\newcommand{\iterMgOuter}{i}%\ensuremath{\nu}}
\newcommand{\iterMgInner}{n}%\ensuremath{\eta}}

\newcommand{\Div}{\nabla\!\cdot\!}

\newcommand{\traceOp}{\ensuremath{\gamma_\level}}
\newcommand{\injectionOp}{\ensuremath{I}}
\newcommand{\projectionOp}{\ensuremath{P^a}}
\newcommand{\projectionOrthogonalOP}{\ensuremath{P^{\langle\cdot\rangle}}}

\newcommand{\skeletalProj}{\ensuremath{\Pi^\partial}}
\newcommand{\averagingOp}{\ensuremath{I^\textup{avg}}}
\newcommand{\contLinProj}{\ensuremath{\overline \Pi^\textup c}}
\newcommand{\discProj}{\ensuremath{\Pi^\textup d}}
\newcommand{\liftingOp}{\ensuremath{S}}

\renewcommand{\vec}[1]{\ensuremath{\boldsymbol{#1}}}
\newcommand{\Nu}{\ensuremath{\vec \nu}}
\newcommand{\dx}{}
\newcommand{\ds}{}
\newcommand{\localU}{\ensuremath{\mathcal U}}
\newcommand{\localV}{\ensuremath{\mathcal V}}
\newcommand{\localQ}{\ensuremath{\vec{\mathcal Q}}}

\newcommand{\avg}[1]{\{\!\{ #1 \}\!\}}

\newcommand{\ureconstructed}{\overline{u}}

\newcommand{\IR}{\ensuremath{\mathbb R}}

\newcommand{\coloneq}{\mathrel{\mathop:}=}

\usepackage[colorlinks = true, linkcolor = blue, citecolor = blue, urlcolor = blue]{hyperref}

\makeatletter
\newcommand\footnoteref[1]{\protected@xdef\@thefnmark{\ref{#1}}\@footnotemark}
\makeatother

\begin{document}

\title[Homogeneous multigrid for hybrid discretizations]{Homogeneous multigrid for hybrid discretizations: application to HHO methods} 

\author[D.\ A.\ Di Pietro]{Daniele A.\ Di Pietro}
\address{IMAG, Univ. Montpellier, CNRS, Montpellier, France}
\email{daniele.di-pietro@umontpellier.fr}

\author[Z.\ Dong]{Zhaonan Dong}
\address{Inria, 2 rue Simone Iff, 75589 Paris, France, 2) CERMICS, Ecole des Ponts, 77455 Marne-la- Vall\'{e}e, France }
\email{zhaonan.dong@inria.fr}

\author[G.\ Kanschat]{Guido Kanschat}
\address{Interdisciplinary Center for Scientific Computing (IWR), Heidelberg University, Mathematikon, Im Neuenheimer Feld 205, 69120 Heidelberg, Germany}
\email{kanschat@uni-heidelberg.de}

\author[P.\ Matalon]{Pierre Matalon}
\address{CMAP, CNRS, École Polytechnique, Institut Polytechnique de Paris, 91120 Palaiseau, France}
\email{pierre.matalon@polytechnique.edu}

% \author[U.\ Rüde]{Ulrich Rüde}
% \address{Department of Computer Science,  Faculty of Engineering, Friedrich-Alexander-Universität Erlangen-Nuremberg, Cauerstrasse 11, 91058 Erlangen, Germany}
% \email{ulrich.ruede@fau.de}

\author[A.\ Rupp]{Andreas Rupp}
\address{School of Engineering Sciences, Lappeenranta--Lahti University of Technology, P.O. Box 20, 53851 Lappeenranta, Finland}
\email{andreas.rupp@fau.de}

% \thanks{%
% }

\subjclass[2010]{65F10, 65N30, 65N50}

\begin{abstract}
 We prove the uniform convergence of the geometric multigrid V-cycle for hybrid high-order (HHO) and other discontinuous skeletal methods. Our results generalize previously established results for HDG methods, and our multigrid method uses standard smoothers and local solvers that are bounded, convergent, and consistent. We use a weak version of elliptic regularity in our proofs. Numerical experiments confirm our theoretical results.
 \\[1ex] \noindent \textsc{Keywords.} Geometric multigrid, homogeneous multigrid, hybrid high-order methods, skeleton methods.
\end{abstract}

\date{\today}
\maketitle

% ----------------------------------------------------------------------------------------------------
\section{Introduction}
% ----------------------------------------------------------------------------------------------------

In the context of fast solvers for linear systems arising from hybrid discretizations of second-order elliptic equations, we propose in this work a generalized framework for the convergence analysis of geometric multigrid methods \cite{briggs_multigrid_2000,trottenberg_multigrid_2001}.

Hybrid discretization methods have been part of the numerical analyst's toolbox to solve partial differential equations since the seventies, starting with hybridized versions of mixed methods such as the Raviart-Thomas (RT-H) \cite{raviart_thomas_1977} and the Brezzi-Douglas-Marini (BDM-H) \cite{brezzi_two_1985} methods. They have gained growing interest in recent years with the outbreak of modern schemes such as the unifying framework of Hybridizable Discontinuous Galerkin (HDG) methods \cite{cockburn_unified_2009} and Hybrid High-Order (HHO) methods \cite{Di-Pietro.Ern.ea:14,Di-Pietro.Ern:15,di_pietro_hybrid_2020}. The list also includes, but is not limited to, Compatible Discrete Operators (CDO) \cite{bonelle_cdo_2014}, Mixed and Hybrid Finite Volumes (MHFV) \cite{droniou_mixed_2006,eymard_discretization_2010,droniou_unified_2010}, Mimetic Finite Differences (MFD) \cite{da_veiga_mimetic_2014}, weak Galerkin \cite{WANG2013103}, local discontinuous Galerkin-hybridizable (LDG-H) \cite{cockburn_ldgh_2008} and discontinuous Petrov--Galerkin (DPG) \cite{dg_dpg_2011} methods. Hybrid discretization methods are characterized by the location of their degrees of freedom (DOFs), placed both within the mesh cells and on the faces. In this configuration, the discrete scheme is built so that the cell DOFs are only locally coupled, leaving the face DOFs in charge of the global coupling. Algebraically, this feature enables the local elimination of the cell unknowns from the arising linear system, resulting in a Schur complement of reduced size, where only the face unknowns remain. The mechanical engineering terminology refers to this elimination process as \emph{static condensation}, and the resulting system is often called the (statically) \emph{condensed} system. The terminology \emph{trace} or \emph{skeleton} system is also employed.
For an extended introduction to hybrid methods and hybridization, we refer to the preface of \cite{di_pietro_hybrid_2020} and the first pages of \cite{cockburn_unified_2009}. 

This paper addresses the fast solution of the condensed systems arising from hybrid discretizations of second-order elliptic equations, focusing on multigrid methods.
Various multigrid solvers and preconditioners have been designed over the last decade, focusing primarily on HDG \cite{cockburn_multigrid_2014,schutz_hierarchical_2017,wildey_unified_2019,muralikrishnan_multilevel_2020,LuRK21,LuRK22a} and HHO \cite{matalon_h-multigrid_2021,di_pietro_towards_2021,botti_p-multilevel_2021,di_pietro_high_order_mg_2023}, but also covering specifically DPG \cite{roberts2016geometric,petrides_adaptive_2021} and RT-H/BDM-H \cite{gopalakrishnan_convergent_2009}.
%% If these solutions all consist of geometric multigrid methods, we can distinguish \cite{di_pietro_algebraic_2023} as a fully algebraic multigrid method.
The above references focus on geometric multigrid methods, while a fully algebraic multigrid method has been considered in \cite{di_pietro_algebraic_2023}.%

The main difficulty in designing a geometric multigrid algorithm for a trace system resides in the fact that the DOFs are supported by the mesh skeleton, which makes classical intergrid transfer operators designed for element-based DOFs unsuitable. While earlier approaches \cite{cockburn_multigrid_2014,gopalakrishnan_convergent_2009,kronbichler_performance_2018} recast trace functions into bulk functions in order to make use of a known efficient solver (typically, a standard piecewise linear continuous Finite Element multigrid solver), the more recent developments seem to converge toward the so-called notion of \emph{homogeneous} multigrid, where the hybrid discretization is conserved at every level of the mesh hierarchy and ``interskeleton’’ transfer operators are designed to communicate trace functions from one mesh skeleton to the other. To this end, multiple skeletal injection operators have already been proposed \cite{LuRK22a,matalon_h-multigrid_2021}.

In this paper, we build upon the work of \cite{LuRK21,LuRK22a,LuRK22b,LuWKR23} on HDG methods to propose a generalized demonstration framework for the uniform convergence of \emph{homogeneous} multigrid methods with V-cycle. The main motivation behind this generalization is to include the HHO methods in its application scope. Indeed, if recent multigrid solvers for HHO methods \cite{matalon_h-multigrid_2021,di_pietro_towards_2021,di_pietro_high_order_mg_2023} have experimentally shown their optimal behavior, the supporting theory is still missing. In the second part of this paper, we then prove that the classical HHO method fulfills the assumptions of this new framework.

The present theory does not assume full elliptic regularity of the problem, thus allowing complex domains with re-entrant corners. However, although modern hybrid methods natively handle polyhedral elements, it is restricted to simplicial meshes. The multigrid method is built in a standard fashion from an abstract injection operator: standard smoothers are used, and the restriction operator is chosen as the adjoint of the injection operator. A classical, symmetric V-cycle is used if the problem exhibits full elliptic regularity. If not, then a variable V-cycle is employed, in which the number of (symmetric) smoothing steps increases as the level decreases in the hierarchy.

The paper is organized as follows. In Section \ref{SEC:pb_formulation}, we introduce notations and the model problem. Section \ref{SEC:disc_skel_methods} describes the abstract framework: (i) the hybrid method is described abstractly, in the form of approximation spaces, local linear operators, stabilization term and bilinear form; (ii) the injection operators used in the multigrid method and its analysis are introduced; (iii) the properties assumed from the hybrid method (\eqref{EQ:LS1}-\eqref{EQ:LS9}) and the injection operator (\eqref{EQ:IA1}-\eqref{EQ:IA2}) are listed.
Section \ref{SEC:convergence_results} describes the multigrid method and asserts the associated convergence results, while Section \ref{SEC:analysis} carries out the convergence analysis.
In Section~\ref{SEC:hho_proofs}, we verify our framework's assumptions for the standard HHO method, and analyze various injection operators in Section~\ref{SEC:hho_injection_operators}.
Finally, Section \ref{SEC:numerical_experiments} presents the numerical experiments supporting our theoretical results, realized with the HHO method on two and three-dimensional test cases.

% ----------------------------------------------------------------------------------------------------
\section{Problem formulation and notation} \label{SEC:pb_formulation}
% ----------------------------------------------------------------------------------------------------
% 
Let $\Omega \subset \IR^d$ be a polygonally bounded Lipschitz domain with boundary $\partial \Omega$. We approximate the unknown $u \in H^1_0(\Omega)$ satisfying
\begin{equation} \label{EQ:poisson_pb}
 - \Delta u = f \qquad \text{ in } \Omega
\end{equation}
for some right-hand side $f$. We assume the problem to be regular in the sense that
\begin{equation}\label{EQ:regularity}
 \| u \|_{1 + \alpha} \le c \| f \|_{\alpha - 1} \qquad \text{ for some } \alpha \in (\tfrac{1}{2},1].
\end{equation}
Here, $\| \cdot \|_\alpha$ for $\alpha \in \IR_{\ge 0}$ denotes the induced norm of the Sobolev--Slobodeckij space $H^\alpha(\Omega)$. By $\| \cdot \|_{-\alpha}$ and $H^{-\alpha}(\Omega)$ we denote the dual norm and the dual space of $H^\alpha_0(\Omega)$ with respect to the extension of $L^2$ duality, respectively. We denote by $\|\cdot\|_X$ and $(\cdot,\cdot)_X$ the $L^2$-norm and $L^2$-scalar product with respect to the set $X\subset \Omega$, respectively.
The same notation is also used for the inner product
of $[L^2(X)]^d$ (the exact meaning can be inferred from the context) such that for all $\vec p, \vec q \in [L^2(X)]^d$, $(p, q)_X \coloneq \int_X \vec p\cdot \vec q$, where $\cdot$ denotes the dot product. Without loss of generality, we assume that $\Omega$ has diameter 1.

\begin{remark}[More general boundary conditions]
 The following approaches directly transfer to non-homogeneous Dirichlet boundary conditions since those only influence the right-hand sides of the numerical schemes. However, this is not true (without some further arguments) for locally changing the type of the boundary condition.
\end{remark}
% 
% ----------------------------------------------------------------------------------------------------
\section{Discontinuous skeletal methods}\label{SEC:disc_skel_methods}
% ----------------------------------------------------------------------------------------------------
% 
We will render a general framework for all discontinuous skeletal methods covering HDG and HHO methods. To this end, we start with a family $(\mesh_\level)_{\level=0,\dots,L}$ of successively refined simplicial meshes and their corresponding face sets $(\faceSet_\level)_{\level=0,\dots,L}$. 
For all level $\level$, we define $h_\level = \max_{\elem\in\mesh_\level} h_\elem$, where $h_\elem$ denotes the diameter of $\elem$.
The family $(\mesh_\level)_{\level=0,\dots,L}$ is regular (see \cite[Def.\ 11.2]{ErnG21v1} for a precise definition), which implies that mesh elements do not deteriorate, and all of its elements are geometrically conforming, which excludes hanging nodes. Additionally, we assume that either $\face \subset \partial \Omega$, or $\face \cap \partial \Omega$ comprises at most one point, and that refinement does not progress too fast, i.e., there is $c_\text{ref} > 0$ such that
\begin{equation}\label{EQ:cref}
  h_\level \ge c_\text{ref} h_{\level-1} \qquad \forall \level \in \{1, \dots, L\}.
\end{equation}
Finally, we assume that our mesh is quasi-uniform implying that there is a constant $c_\textup{uni}$ such that $h_\level \le c_\textup{uni} \min_{\elem \in \mesh_\level} h_\elem$.

Next, we choose a finite-dimensional space $M_\face \subset L^2(\face)$ of functions living on face $\face$ for all $\face \in \faceSet_\level$. In the numerical scheme, this space will be the test and trial spaces for the skeletal variable $m_\level$ approximating the trace of the unknown $u$ on the skeleton, which we denote by $\faceSet_\level$. Moreover, finite-dimensional approximation spaces $V_\elem \subset L^2(\elem)$ and $\vec W_\elem \subset L^2(\elem)^d$ are defined element-wise. The space $V_\elem$ is the local test and trial space for the primary unknown $u$, while the space $\vec W_\elem$ is the local test and trial space for the dual unknown $\vec q = - \nabla u$. Global spaces are defined by concatenation via
\begin{align}
  \skeletalSpace_\level &\coloneq \left\{ m \in L^2 (\faceSet_\level) \;\middle|\;
  \begin{array}{r@{\,}c@{\,}ll}
  m_{|\face} &\in& \polynomials_p & \forall \face \in \faceSet_\level\\
  m_{|\face} &=& 0 & \forall \face \subset \partial \Omega   
  \end{array} \right\}, \\
  \discElementSpace_\level &\coloneq\bigl\{ v \in L^2(\Omega) \big|\;v_{|\elem} \in V_\elem, \forall \elem \in \mesh_\level \bigr\},\\
  \vec W_\level &\coloneq\bigl\{ \vec q \in L^2(\Omega;\mathbb R^d) \big|\;\vec q_{|\elem} \in \vec W_\elem, \forall \elem \in \mesh_\level \bigr\}.    
\end{align}

We define element-wise, abstract linear operators
% 
% \begin{align*}
%  \localU_\elem & \colon \skeletalSpace_\level|_{\partial \elem} \to \discElementSpace_\elem, & \localV_\elem & \colon L^2(\elem) \to \discElementSpace_\elem, \\
%  \localQ_\elem & \colon \skeletalSpace_\level|_{\partial \elem} \to \vec W_\elem, & \vec{\mathcal R}_\elem & \colon L^2(\elem) \to \vec W_\elem.
% \end{align*}
\begin{align*}
 \localU_\elem \colon \skeletalSpace_\level|_{\partial \elem} \to \discElementSpace_\elem, && \localV_\elem \colon L^2(\elem) \to \discElementSpace_\elem, &&
 \localQ_\elem \colon \skeletalSpace_\level|_{\partial \elem} \to \vec W_\elem.%, & \vec{\mathcal R}_\elem & \colon L^2(\elem) \to \vec W_\elem.
\end{align*}
The choice of these operators, called local solvers, influences the numerical schemes. Global linear operators are constructed by concatenation of their element-local analogous:
% 
% \begin{align*}
%  \localU_\level & \colon \skeletalSpace_\level \to \discElementSpace_\level, & \localV_\level & \colon L^2(\Omega) \to \discElementSpace_\level, \\
%  \localQ_\level & \colon \skeletalSpace_\level \to \vec W_\level, & \vec{\mathcal R}_\level & \colon L^2(\Omega) \to \vec W_\level.
% \end{align*}
\begin{align*}
 \localU_\level \colon \skeletalSpace_\level \to \discElementSpace_\level, && \localV_\level \colon L^2(\Omega) \to \discElementSpace_\level, &&
 \localQ_\level \colon \skeletalSpace_\level \to \vec W_\level.%, & \vec{\mathcal R}_\level & \colon L^2(\Omega) \to \vec W_\level.
\end{align*}
One can show that $m_\level$ approximates the trace of $u$ on the skeleton and that
\begin{equation*}
 u \approx \localU_\level m_\level + \localV_\level f%, \quad \text{ and } \quad \vec q \approx \localQ_\level m_\level %+ \vec{\mathcal R}_\level f
\end{equation*}
in the bulk of $\Omega$ if $m_\level$ satisfies 
\begin{equation} \label{EQ:condensed_pb}
 a_\level(m_\level,\mu) = \int_\Omega f \,\localU_\level \mu \dx \qquad \text{ for all } \mu \in \skeletalSpace_\level,
\end{equation}
where the elliptic and continuous bilinear form $a_\level$ has the form
\begin{equation} \label{EQ:bilinear}
 a_\level(m,\mu) = \int_\Omega \localQ_\level m \cdot \localQ_\level \mu \dx + s_\level(m,\mu).
\end{equation}
The symmetric positive semi-definite $s_\level$ is usually referred to as (condensed) \emph{penalty} or \emph{stabilizing term}.

The choices of the spaces $M_\face$, $V_\elem$, $\vec W_\elem$, the local solvers $\localU_\level$, $\localV_\level$, $\localQ_\level$ and the stabilizing term $s_\level$ completely define a discontinuous skeletal method.
\begin{remark}[Possible choices of hybrid methods]\
 \begin{itemize}[leftmargin=*]
 \item For the \textbf{LDG-H} method, we set $M_\face = \mathcal P_p(\face)$, $V_\elem = \mathcal P_p(\elem)$, and $\vec W_\elem = \mathcal P^d_p(\elem)$.
   The operators $\localU_\elem$ and $\localQ_\elem$ map $m_{\partial \elem}$ to the element-wise solutions $u_\elem\in V_\elem$ and $\vec q_\elem \in \vec W_\elem$ of 
  \begin{subequations}\label{EQ:hdg_scheme}
  \begin{align}
   \int_\elem \vec q_\elem \cdot \vec p_\elem \dx - \int_\elem u_\elem \Div \vec p_\elem \dx
   & = - \int_{\partial \elem} m \vec p_\elem \cdot \Nu \ds\label{EQ:hdg_primary} \\
   \int_{\partial \elem} ( \vec q_\elem \cdot \Nu + \tau_\level u_\elem ) v_\elem \ds - \int_\elem \vec q_\elem \cdot \nabla v_\elem \dx
   & = \tau_\level \int_{\partial \elem} m v_\elem \ds \label{EQ:hdg_flux}
  \end{align}
  \end{subequations}
  with test functions $v_\elem\in V_\elem$ and $\vec p_\elem \in \vec W_\elem$, and $\Nu$ the outward normal vector to $\partial\elem$. In this sense $\localU_\elem \colon m_{\partial \elem} \mapsto u_\elem$, $\localQ_\elem: m_{\partial \elem} \mapsto \vec q_\elem$ for all $\elem \in \mesh_\level$.

  Finally, the bilinear form $a_\level$ is defined as
  \begin{equation*}
   a_\level(m_\level,\mu) = \int_\Omega \localQ_\level m_\level \cdot \localQ_\level \mu \dx + \underbrace {\sum_{\elem \in \mesh_\level} \tau_\level \int_{\partial \elem} (\localU_\level m_\level - m_\level) (\localU_\level \mu - \mu) \ds }_{ = s_\level(m_\level,\mu) }
  \end{equation*}
  for some parameter $\tau_\level > 0$.
  \item For the \textbf{RT-H} and \textbf{BDM-H} methods, one uses the framework of the LDG-H method replacing the local bulk spaces by the RT-H and BDM-H spaces with $\tau_\level = 0$.
  \item For the \textbf{HHO} method, we set $M_\face = \mathcal P_p(\face)$, $V_\elem = \mathcal P_p(\elem)$, and $\vec W_\elem = \nabla \mathcal P_{p+1}(\elem)$.
  For $(u_\elem, m_{\partial \elem}) \in V_\elem \times M_{\partial\elem}$, we define $\vec q_\elem (u_\elem, m_{\partial \elem})$ as the element-wise solution of
  \begin{equation} \label{EQ:hho_flux}
   \int_\elem \vec q_\elem(u_\elem, m_{\partial \elem}) \cdot \vec p_\elem \dx - \int_\elem u_\elem \Div \vec p_\elem \dx 
   = - \int_{\partial \elem} m_{\partial \elem} \vec p_\elem \cdot \Nu \ds
  \end{equation}
  for all $\vec p_\elem \in \vec{W}_\elem$.
  Note that this relation is the same as \eqref{EQ:hdg_primary}. In HHO terminology, $\vec q_\elem(u_\elem, m_{\partial \elem})$ corresponds to $-\nabla \theta_\elem^{p+1}(u_\elem, m_{\partial \elem})$, where $\theta_\elem^{p+1}$ denotes the so-called local higher-order reconstruction operator, defined by \eqref{EQ:hho_flux} and the closure condition
  \begin{equation} \label{EQ:hho_closure}
   \int_\elem \theta_\elem^{p+1}(u_\elem, m_{\partial \elem}) \dx = \int_{\partial \elem} u_\elem \dx.
  \end{equation}
  In a second step, given a local bilinear stabilizer $\underline{s}_\elem((u_\elem,m_{\partial \elem}), (v_\elem,\mu))$, we introduce the local bilinear form
  \begin{multline} \label{EQ:hho_bilinear_uncondensed}
   \underline{a}_\elem((u_\elem, m_{\partial \elem}), (v_\elem, \mu)) \\
   = \int_\elem \vec q_\elem(u_\elem, m_{\partial \elem}) \cdot \vec q_\elem(v_\elem,\mu) \dx + \underline{s}_\elem ((u_\elem, m_{\partial \elem}), (v_\elem, \mu)).
  \end{multline}
  Consider the following local problems:\newline
  (i) For $m_{\partial \elem} \in M_{\partial\elem}$, find $u_\elem^1 \in V_\elem$ such that
\begin{equation} \label{EQ:hho_localU}
    \underline{a}_\elem((u_\elem^1, 0), (v_\elem, 0)) = -\underline{a}_\elem((0, m_{\partial \elem}), (v_\elem, 0))
    \qquad
    \forall v_\elem \in V_\elem.
\end{equation}
(ii) For $f\in L^2(\Omega)$, find $u_\elem^2 \in V_\elem$ such that
\begin{equation} \label{EQ:hho_localV}
    \underline{a}_\elem((u_\elem^2, 0), (v_\elem, 0)) = \int_\elem f v_\elem \dx
    \qquad
    \forall v_\elem \in V_\elem.
\end{equation}
We define $\localU_\elem\colon m_{\partial \elem} \mapsto u_\elem^1$ solution of \eqref{EQ:hho_localU}, $\localV_\elem\colon f \mapsto u_\elem^2$ solution of \eqref{EQ:hho_localV}, and we set element-by-element
\begin{equation}\label{eq:localQ}
  \localQ_\elem\colon m_{\partial \elem} \mapsto \vec q_\elem(\localU_\elem m_{\partial \elem}, m_{\partial \elem})
\end{equation}

  Finally, the global bilinear form is defined as
  \begin{equation} \label{EQ:hho_bilinear}
\begin{aligned}
 a_\level(m_\level,\mu) 
    =:&
 \sum_{\elem \in \mesh_\level}    \underline{a}_\elem((\localU_\elem m_{\partial \elem}, m_{\partial \elem}), (\localU_\elem \mu, \mu)) \\
= &
    \int_\Omega \localQ_\level m_\level \cdot \localQ_\level \mu \dx  + 
    \underbrace{\sum_{\elem \in \mesh_\level} \underline{s}_\elem ((\localU_\level m_\level, m_\level), (\localU_\level \mu, \mu)) }_{ = s_\level(m_\level,\mu)}.
\end{aligned}
  \end{equation}
 \end{itemize}
\end{remark}

%%%%%%%%%%%%%%%%%%%%%%%%%%%%%%%%%%%%%%%%%%%%%%%%%%%%%%%%%%%%%%%%%%%%%% 
\subsection{Operators for the multigrid method and analysis}
%%%%%%%%%%%%%%%%%%%%%%%%%%%%%%%%%%%%%%%%%%%%%%%%%%%%%%%%%%%%%%%%%%%%%%
% 
For functions $\rho, \mu \in L^2(\faceSet_\level)$, we set
\begin{equation*}
 \langle \rho, \mu \rangle_\level = \sum_{\elem \in \mesh} \frac{|\elem|}{|\partial \elem|} \int_{\partial \elem} \rho \mu \ds.
\end{equation*}
Importantly, $\langle \cdot, \cdot \rangle_\level$ defines a scalar product on $\skeletalSpace_\level$, and the norm of this scalar product
\[
\| \cdot \|_\level \coloneq \langle \cdot, \cdot \rangle_\level^{\frac12}
\]
scales like the $L^2$-norm in the bulk of the domain. Moreover, the scalar product can readily be applied to a combination of bulk and skeleton functions: For example, $\langle \mu, v \rangle_\level$ can readily be evaluated for $\mu \in \skeletalSpace_\level$ and $v \in V_\level$ with the understanding that the trace of $v$ on $\partial T$ is used inside the integral. We relate the bilinear forms $a_\level$ and $\langle \cdot, \cdot \rangle_\level$ using the operator $A_\level$, which is defined via
\begin{equation}\label{eq:Al}
  \langle A_\level \rho, \mu \rangle_\level = a_\level(\rho, \mu) \qquad \forall \rho, \mu \in \skeletalSpace_\level.
\end{equation}

The injection operator $\injectionOp_\level \colon \skeletalSpace_{\level - 1} \to \skeletalSpace_\level$ remains abstract at this level. Its properties securing our analytical findings are listed in Section \ref{sec:ass-injection}. Possible realizations of injection operators can be found in \cite{LuRK22a}. The multigrid operator $B_\level\colon \skeletalSpace_\level \to \skeletalSpace_\level$ for preconditioning $A_\level$ will be defined in Section \ref{SEC:main_convergence_result}.

The operators $\projectionOrthogonalOP_{\level-1}$ and $\projectionOp_{\level-1}$, which are defined via
\begin{xalignat}3
 \projectionOrthogonalOP_{\level-1}&\colon \skeletalSpace_\level \to \skeletalSpace_{\level-1},
 &\langle \projectionOrthogonalOP_{\level-1} \rho, \mu \rangle_{\level-1}
 &= \langle \rho, \injectionOp_\level \mu\rangle_\level
 && \forall \mu \in \skeletalSpace_{\level-1},
 \label{EQ:L2_projection_definition}
 \\
 \projectionOp_{\level-1}&\colon \skeletalSpace_\level \to \skeletalSpace_{\level-1},
 &a_{\level-1}(\projectionOp_{\level-1} \rho, \mu)
 &= a_\level(\rho, \injectionOp_\level \mu)
 && \forall \mu \in \skeletalSpace_{\level-1},
 \label{EQ:projection_definition}
\end{xalignat}
replace the $L^2$ and the Ritz projections of conforming methods, respectively. While the former (or a discrete variation of it) is relevant for implementing multigrid methods, the latter is key to the analysis. We also introduce the $L^2$-projections
\begin{align*}
 \skeletalProj_\level \colon & H^1_0(\Omega) \to \skeletalSpace_\level, && \langle \skeletalProj_\level v , \mu \rangle_\level = \langle v, \mu \rangle_\level & \forall \mu \in \skeletalSpace_\level,\\
 \discProj_\level \colon & H^1(\Omega) \to \discElementSpace_\level, && (\discProj_\level v, w)_\Omega = (v,w)_\Omega & \forall w \in \discElementSpace_\level.
\end{align*}
%
%%%%%%%%%%%%%%%%%%%%%%%%%%%%%%%%%%%%%%%%%%%%%%%%%%%%%%%%%%%%%%%%%%%%%% 
%%%%%%%%%%%%%%%%%%%%%%%%%%%%%%%%%%%%%%%%%%%%%%%%%%%%%%%%%%%%%%%%%%%%%% 
\subsection{Assumptions on discontinuous skeletal methods}
\label{sec:ass-local}
%%%%%%%%%%%%%%%%%%%%%%%%%%%%%%%%%%%%%%%%%%%%%%%%%%%%%%%%%%%%%%%%%%%%%%
%%%%%%%%%%%%%%%%%%%%%%%%%%%%%%%%%%%%%%%%%%%%%%%%%%%%%%%%%%%%%%%%%%%%%%
% 
Here and in the following, $\lesssim$ means smaller than or equal to, up to a constant independent of the mesh size $h_\level$ and the multigrid level $\level$.
We also write $A \simeq B$ as a shortcut for ``$A \lesssim B$ and $B \lesssim A$''.
We assume that the numerical hybrid method, characterized in Section \ref{SEC:disc_skel_methods}, satisfies the following conditions for any $\mu \in \skeletalSpace_\level$:
\begin{itemize}
 \item The trace of the bulk unknown approximates the skeletal unknown:
 \begin{equation}
  \| \localU_\level \mu - \mu \|_\level \lesssim h_\level \| \mu \|_{a_\level}, \tag{HM1}\label{EQ:LS1}
 \end{equation}
 where $\| \cdot \|_{a_\level}$ denotes the norm induced by $a_\level$ on $M_\level$, i.e.,
 \begin{equation}\label{eq:norm.al}
   \| \cdot \|_{a_\level} \coloneq a_\level(\cdot, \cdot)^{\frac12}.
 \end{equation}
 \item The operators $\localQ_\level \mu$ and $\localU_\level \mu$ are continuous:
 \begin{equation}
  \| \localQ_\level \mu \|_\Omega \lesssim h^{-1}_\level \| \mu \|_\level \quad \text{ and } \quad \| \localU_\level \mu \|_\Omega \lesssim \| \mu \|_\level. \tag{HM2}\label{EQ:LS2}
 \end{equation}
\item The dual approximation $\localQ_\level \mu \sim -\nabla_\level \localU_\level \mu$, where $\nabla_\level$ denotes the broken gradient.
  That is,
 \begin{equation}
  \| \localQ_\level \mu + \nabla_\level \localU_\level \mu \|_\Omega \lesssim h_\level^{-1} \| \localU_\level \mu - \mu \|_\level. \tag{HM3}\label{EQ:LS3}
 \end{equation}
 \item Consistency with the standard linear finite element method: If $w \in \linElementSpace_\level$, we have
 \begin{equation}\label{EQ:LS4}
  \localU_\level \gamma_\level w = w \qquad \text{ and } \qquad \localQ_\level \gamma_\level w = - \nabla w, \tag{HM4}
 \end{equation}
 where $\gamma_\level$ is the trace operator to the skeleton $\faceSet_\level$ and
 \begin{equation*}
 \linElementSpace_\level \coloneq \bigl\{ v \in H^1_0(\Omega) \; \big| \; v_{|\elem} \in \polynomials_1(\elem) \;\; \forall \elem \in \mesh_\level\bigr\}.
 \end{equation*}
 \item Convergence of the skeletal unknown to the traces of the analytical solution. That is, if $m_\level$ is the skeletal function of the hybrid approximation to $u \in H^{1+\alpha}(\Omega)$, we have
 \begin{equation}
  \| m_\level - \skeletalProj_\level u\|_{a_\level}  \lesssim h_\level^{\alpha} \| u\|_{\alpha + 1}.\tag{HM5}\label{EQ:LS5}
 \end{equation}
 \item The usual bounds on the eigenvalues of the condensed discretization matrix hold:
 \begin{equation}
  \| \mu \|^2_\level \lesssim a_\level(\mu,\mu) \lesssim h^{-2}_\level \| \mu \|^2_\level. \tag{HM6}\label{EQ:LS6}
 \end{equation}
 \item If $\rho = \gamma_\level w$ for $w \in \linElementSpace_\level$, the stabilization satisfies
 \begin{equation}\label{EQ:LS7}
  s_\level(\rho, \mu) = 0. \tag{HM7}
 \end{equation}
 \item The global bilinear form is bounded by an HDG-type norm:
 \begin{equation}
  a_\level(\mu,\mu) \lesssim \| \localQ \mu \|^2_\Omega + \sum_{\elem \in \mesh_\level} \tfrac1{h_\elem} \| \localU \mu - \mu \|_{\partial \elem}^2. \tag{HM8}\label{EQ:LS9}
 \end{equation}
\end{itemize}
\begin{remark}[HDG methods]
 For the LDG-H (with $\tau_\level h_\level \lesssim 1$), the RT-H, and the BDM-H methods, \eqref{EQ:LS7} follows directly from \eqref{EQ:LS4}, and \eqref{EQ:LS9} follows directly from the definition of the bilinear form $a_h$. They use almost identical assumptions to prove multigrid convergence for the HDG, BDM-H, and RT-H methods. \eqref{EQ:LS2}--\eqref{EQ:LS6} have been used to prove the convergence of multigrid methods for HDG, while \eqref{EQ:LS1}, \eqref{EQ:LS7}, and \eqref{EQ:LS9} are novel assumptions that allow to generalize the convergence theory to HHO.
\end{remark}
% 
%%%%%%%%%%%%%%%%%%%%%%%%%%%%%%%%%%%%%%%%%%%%%%%%%%%%%%%%%%%%%%%%%%%%%% 
\subsection{Assumptions on injection operators}
\label{sec:ass-injection}
%%%%%%%%%%%%%%%%%%%%%%%%%%%%%%%%%%%%%%%%%%%%%%%%%%%%%%%%%%%%%%%%%%%%%%
% 
Our analytical findings will rely on the following conditions to hold on all mesh levels $\level$:
\begin{enumerate}
 \item Boundedness:
 \begin{equation}\tag{IA1}\label{EQ:IA1}
  \| \injectionOp_\level \rho \|_\level \lesssim \| \rho \|_{\level - 1} \qquad \forall \rho \in \skeletalSpace_{\level - 1}.
 \end{equation}
 \item Conformity with linear finite elements:
 \begin{equation}\tag{IA2}\label{EQ:IA2}
  \injectionOp_\level \gamma_{\level - 1} w = \gamma_\level w  \qquad \forall w \in \linElementSpace_{\level - 1}.
 \end{equation}
\end{enumerate}
This way, we do not restrict ourselves to one specific injection operator. All injection operators in \cite{LuRK22a} satisfy the desired properties. An immediate consequence of \eqref{EQ:LS4} and \eqref{EQ:IA2} is
\begin{lemma}[Quasi-orthogonality]\label{LEM:quasi_orth}
 Let \eqref{EQ:LS4}, \eqref{EQ:LS7}, and \eqref{EQ:IA2} hold, then we have for all $\mu \in \skeletalSpace_{\level}$ and $w \in \linElementSpace_{\level - 1}$ that
 \begin{equation*}
  (\localQ_\level \mu - \localQ_{\level - 1} \projectionOp_{\level - 1} \mu, \nabla w)_\Omega = 0.
 \end{equation*}
\end{lemma}
\begin{proof}
 For $\rho = \gamma_{\level - 1} w$, assumptions \eqref{EQ:IA2} and \eqref{EQ:LS4}, we obtain  using the embedding $\linElementSpace_{\level - 1} \subset \linElementSpace_\level$
 \begin{gather*}
  \localQ_{\level - 1} \rho = \localQ_{\level - 1} \gamma_{\level - 1} w = - \nabla w = \localQ_\level \gamma_\level w = \localQ_\level \injectionOp_\level \gamma_{\level - 1} w = \localQ_\level \injectionOp_\level \rho.
  % \localU_{\level - 1} \gamma_{\level - 1} w = w = \localU_\level \gamma_\level w = \localU_\level \injectionOp_\level \gamma_{\level - 1} w
 \end{gather*}
 The definitions of $a_\level$, $a_{\level - 1}$, and \eqref{EQ:LS7} yield
 \begin{align*}
  a_{\level - 1}(\projectionOp_{\level - 1} \mu, \rho) &= (\localQ_{\level - 1} \projectionOp_{\level - 1} \mu, \localQ_{\level - 1} \rho)_\Omega, \\
  a_{\level}(\mu, \injectionOp_\level \rho) &= (\localQ_{\level} \mu,  \localQ_{\level} \injectionOp_\level \rho)_\Omega.
 \end{align*}
 The two lines are equal by definition \eqref{EQ:projection_definition} of the projection operator $\projectionOp_{\level - 1}$. Thus, taking the difference gives the result.
\end{proof}
\section{Multigrid algorithm and abstract convergence results} \label{SEC:convergence_results}
If \eqref{EQ:poisson_pb} has full elliptic regularity, which means that \eqref{EQ:regularity} holds for $\alpha = 1$, we use a standard (symmetric) V-cycle multigrid method. 
Otherwise, we use a variable V-cycle multigrid method to solve the system of linear equations arising from \eqref{EQ:condensed_pb}. 

We follow the lines of \cite{LuRK22a} and avoid a discussion about good smoothing operators. That is, we allow our smoother in smoothing step $i$
\begin{equation*}
 R^i_\level \colon \skeletalSpace_\level \to \skeletalSpace_\level
\end{equation*}
to fit the criteria of \cite{BrambleP1992}, which allow pointwise Jacobi and Gauss-Seidel methods. 
Next, we present the multigrid method as in \cite{BramblePX1991}. Afterwards, we present the abstract convergence results of \cite{DuanGTZ2007}.
% 
% ---------------------------------------------------------------------
\subsection{Multigrid algorithm}\label{SEC:multigrid_algorithm}
% ---------------------------------------------------------------------
% 
We recursively define the multigrid operator of the refinement level $\level$
\begin{equation*}
 B_\level \colon \skeletalSpace_\level \ni \mu \mapsto B_\level \mu \in \skeletalSpace_\level
\end{equation*}
with $\iterMgInner_{\level} \in \mathbb N \setminus \{ 0 \}$ smoothing steps on level $\level$: Let $B_0 = A^{-1}_0$ be the exact inverse of $A_\level$ on the coarsest level. For $\level > 0$, set $x^0 = 0 \in \skeletalSpace_\level$.
\begin{enumerate}
 \item Perform $\iterMgInner_{\level}$ smoothing steps
 \begin{equation*}
  x^\iterMgOuter = x^{\iterMgOuter-1} + R_\level^{\iterMgOuter} ( \mu - A_\level x^{\iterMgOuter-1} ).
 \end{equation*}
 \item Perform recursive multigrid step
 \begin{equation*}
  q = B_{\level-1} \projectionOrthogonalOP_{\level-1} ( \mu - A_\level x^{\iterMgInner_\level}),
 \end{equation*}
 and set $y^0 = x^{n_\level} + \injectionOp_\level q$.
 \item Perform $\iterMgInner_{\level}$ smoothing steps
 \begin{equation*}
  y^\iterMgOuter = y^{\iterMgOuter - 1} + R^{\iterMgOuter+\iterMgInner_\level}_\level ( \mu - A_\level y^{\iterMgOuter-1} ).
\end{equation*}
\item Set $B_\level \mu = y^{\iterMgInner_\level}$.
\end{enumerate}

If $\iterMgInner_\level = \iterMgInner$ independent of the level, we obtain the standard V-cycle. The variable V-cycle is characterized by
\begin{gather}\label{EQ:variable-v}
 \rho_1\iterMgInner_{\level} \le  \iterMgInner_{\level-1} \le \rho_2\iterMgInner_{\level},
\end{gather}
which needs to hold for all $\level \le L$ and uniform constants $1<\rho_1\le\rho_2$.
%                                                                                                                                                                                                                                                   
% ---------------------------------------------------------------------                                                                                                                                                                 
\subsection{Main convergence results}\label{SEC:main_convergence_result}
% ---------------------------------------------------------------------                                                                                                                                                                                                                                                                                                                             
%
We use the results obtained in \cite{BramblePX1991,DuanGTZ2007}. These publications show convergence under three abstract assumptions, which are assumed to hold for all level $\level$. In order to illustrate these assumptions, let $\lambda^A_\level$ be the largest eigenvalue of $A_\level$, and
\begin{gather*}
 K_\level \coloneq \bigl(\mathds 1 - (\mathds 1 - R_\level A_\level) (\mathds 1 - R'_\level A_\level)\bigr) A^{-1}_\level,
\end{gather*}
where $\mathds 1$ is the identity matrix and $R'_\level$ is the transpose of $R_\level$. The aforementioned, abstract assumptions claim existence of constants $C_1, C_2, C_3 > 0$ independent of the mesh level $\level$ and of the function $\mu \in \skeletalSpace_\level$ satisfying:
\begin{itemize}
\item Regularity of approximation:
 \begin{equation}\label{EQ:precond1a}
  | a_\level (\mu - \injectionOp_\level \projectionOp_{\level-1} \mu, \mu) | \le C_1 \left( \frac{\| A_\level \mu \|^2_\level}{\lambda^A_\level} \right)^\alpha a_\level(\mu, \mu)^{1- \alpha}.\tag{A1}
 \end{equation}
 If $\alpha=1$ in \eqref{EQ:regularity}, \eqref{EQ:precond1a} simplifies to
 \begin{equation}\label{EQ:precond1}
  | a_\level(\mu - \injectionOp_\level \projectionOp_{\level-1} \mu, \mu) |
  \le C_1 \frac{\| A_\level \mu \|^2_\level}{\lambda^A_\level}. \tag{A1'}
 \end{equation}
\item Boundedness of the composition $\injectionOp_{\level} \circ \projectionOp_{\level-1}\colon M_{\level} \to M_{\level}$ of the injection and Ritz quasi-projection operators:
 \begin{equation}\label{EQ:precond2}
  \| \mu - \injectionOp_\level \projectionOp_{\level-1} \mu \|_{a_\level} \le C_2 \| \mu \|_{a_\level}. \tag{A2}
\end{equation}
\item Smoothing hypothesis:
 \begin{equation}\label{EQ:precond3}
  \frac{\| \mu \|^2_\level}{\lambda^A_\level} \le C_3 \langle K_\level \mu, \mu \rangle_\level. \tag{A3}
 \end{equation}
\end{itemize}
\begin{theorem}\label{TH:main_theorem}
  Let \eqref{EQ:regularity} hold with $\alpha = 1$, as well as \eqref{EQ:precond1}, \eqref{EQ:precond2}, and \eqref{EQ:precond3}.
  Then, for the standard V-cycle, for all $\level \ge 0$, and for all $\mu \in \skeletalSpace_\level$
 \begin{equation*}
  | a_\level ( \mu - B_\level A_\level \mu, \mu ) | \le \delta a_\level(\mu, \mu),
 \end{equation*}
 where
 \begin{equation*}
  \delta = \frac{C_1 C_3}{\iterMgInner - C_1 C_3} \qquad \text{with} \qquad \iterMgInner > 2 C_1 C_3.
 \end{equation*}
\end{theorem}
\begin{proof}
 This result is an immediate consequence of~\cite[Thm.~3.1]{DuanGTZ2007}.
\end{proof}

\begin{theorem}\label{TH:variable}
 Let \eqref{EQ:regularity} hold for $\alpha \in (\tfrac 1 2, 1]$. Further, assumptions \eqref{EQ:precond1a}, \eqref{EQ:precond2}, and \eqref{EQ:precond3} shall be true. If, additionally,~\eqref{EQ:variable-v} holds with $\rho_1,\rho_2>1$, then we have, for all $\level \ge 0$ and all $\mu \in \skeletalSpace_\level$, that
  \begin{equation*}
  \eta_0 a_\level ( \mu, \mu )
   \le a_\level ( B_\level A_\level \mu, \mu )
   \le \eta_1 a_\level ( \mu, \mu )
 \end{equation*}
 holds with
  \begin{equation*}
  \eta_0 \ge \frac{\iterMgInner_\level^\alpha}{M+\iterMgInner_\level^\alpha},
  \qquad
  \eta_1 \le\frac{M+\iterMgInner_\level^\alpha}{\iterMgInner_\level^\alpha},
 \end{equation*}
 where the constant $M > 0$ is independent of $\level$. Hence, the condition number of $B_\level A_\level$ does not depend on $\level$.
\end{theorem}
\begin{proof}
This result is an immediate consequence of~\cite[Thm.~6]{BramblePX1991}.
\end{proof}
% 
% ---------------------------------------------------------------------
\section{Convergence analysis}\label{SEC:analysis}
% ---------------------------------------------------------------------
% 
We follow the lines of \cite{LuRK22a} and extend their results to our more general framework. Thus, the theorems and lemmas in the following sections will be very similar to the ones in \cite{LuRK22a}. If Lu et al. \cite{LuRK22a} have not used any HDG-specific arguments in their proofs, we also accept the results to hold in our framework and cite them. Otherwise, we will redo the respective proofs. 
% 
% ---------------------------------------------------------------------
\subsection{Energy boundedness of the injection and proof of~\eqref{EQ:precond2}}
% ---------------------------------------------------------------------
%
We provide a boundedness result for $\injectionOp_\level$ concerning several measures. In particular, we investigate the injection operator's energy boundedness and the ``Ritz projection'' to show \eqref{EQ:precond2}.
To this end, we use the averaging linear interpolation 
\begin{equation}\label{eq:averagingOp}
  \averagingOp_\level\colon\discElementSpace_\level \to \linElementSpace_\level, \qquad     \averagingOp_\level u ({\vec x}_{\vec a}) \coloneq \avg{u}_{\vec a} \text{ for any internal mesh vertex $\vec a$}.
\end{equation}
Here, $\vec x_{\vec a}$ is the point that corresponds to vertex $\vec a$, and $\avg{u}_{\vec a}$ describes the arithmetic mean of all values that $u$ attains in $\vec a$. For $\vec x\in \partial\Omega$, we set $\averagingOp_\level u (\vec x) \coloneq 0$.
\begin{lemma}\label{LEM:avg_bound_by_localQ}
 Assuming \eqref{EQ:LS1} and \eqref{EQ:LS3}, we have for all $\level$ that 
 \begin{align}
  \| \nabla \averagingOp_\level \localU_\level \mu \|_\Omega ~\lesssim~ & \| \mu \|_{a_\level}, && \forall \mu \in \skeletalSpace_\level,\label{EQ:avg_stability} \\
  \| \localU_\level \lambda - \averagingOp_\level \localU_\level \lambda\|_\Omega ~\lesssim~ & h_\level \| \mu \|_{a_\level} && \forall \lambda \in \skeletalSpace_\level.\label{EQ:avg_error}
 \end{align}
\end{lemma}
\begin{proof}
 This can be shown analogously to \cite[Lem.\ 5.2]{LuRK22a}. However, their (LS1) needs to be replaced by our version of \eqref{EQ:LS1}, which changes the right-hand sides in \eqref{EQ:avg_stability} and \eqref{EQ:avg_error}.
\end{proof}

\begin{lemma}\label{LEM:diff_bound_by_q}
 Assuming \eqref{EQ:LS1} and \eqref{EQ:LS3}, we have
 \begin{equation*}
  \| \mu - \gamma_\level \averagingOp_\level \localU_\level \mu \|_\level \lesssim h_\level \| \mu \|_{a_\level} \qquad \forall \mu \in \skeletalSpace_\level.
 \end{equation*}
\end{lemma}
\begin{proof}
 This is \cite[Lem.\ 5.3]{LuRK22a}, but with our version of \eqref{EQ:LS1}.
\end{proof}

\begin{lemma}\label{TH:conv_result}
 Let \eqref{EQ:IA1}, \eqref{EQ:IA2}, \eqref{EQ:LS1}--\eqref{EQ:LS4} hold. We have
 \begin{align}
  \| \localQ_\level \injectionOp_\level \mu \|_\Omega ~\lesssim~ & \| \mu \|_{a_{\level -1}}, \label{EQ:Q_mehses}\\
  \| \localU_{\level - 1} \mu - \localU_\level \injectionOp_\level \mu \|_\Omega ~\lesssim~ & h_{\level - 1} \| \mu \|_{a_{\level - 1}} \lesssim h_\level \| \mu \|_{a_{\level - 1}}\label{EQ:second_ineq}
 \end{align}
 for all $\mu \in \skeletalSpace_{\level - 1}$.  If, additionally, \eqref{EQ:LS9} holds, we have
 \begin{align} 
   \| \injectionOp_\level \mu \|_{a_\level}
   & \lesssim
   \| \mu \|_{a_{\level-1}}
   && \forall \mu \in \skeletalSpace_{\level - 1},\label{EQ:i_stable}\\
   \| \projectionOp_{\level - 1} \mu \|_{a_{\level-1}}
   & \lesssim
     \| \mu \|_{a_\level}
   && \forall \mu \in \skeletalSpace_\level.\label{EQ:ap_stable}
 \end{align}
\end{lemma}
To prove this lemma, we use the lifting operator $\liftingOp_\level\colon \skeletalSpace_\level \to V_\textup{disc}$ (with $V_\textup{disc} \subset C(\Omega) \cap L^2(\Omega)$ being a suitable discrete space). This operator is inspired from \cite{TanPhD}. It is also used in \cite[Lem.\ A.3]{GiraultR1986} for the two dimensional case and in \cite[Def.\ 5.46]{Monk2003} for three dimensional settings:
\begin{subequations}
\label{eq:define-s}
\begin{align}
\label{eq:define-s-a}
 (\liftingOp_\level \mu , v)_\elem &= (\localU_\level \mu, v)_\elem && \forall v \in \polynomials_p(\elem), \, \forall \elem \in \mesh_\level,\\
 \label{eq:define-s-face}
 \langle \liftingOp_\level \mu, \eta \rangle_\face &= \langle \mu, \eta \rangle_\face && \forall \eta \in \polynomials_{p+1}(\face), \, \forall \face \in \faceSet_\level,\\
 \liftingOp_\level \mu(\vec a) &= \avg{ \mu }_{\vec a} && \forall \vec a \text{ is a vertex in } \mesh_\level, \, \vec a \not\in \partial \Omega, \label{eq:define-s-internal-vertex}\\
 \liftingOp_\level \mu(\vec a) &= 0 && \forall \vec a \text{ is a vertex in } \mesh_\level, \, \vec a \in \partial \Omega
\end{align}
in two dimensions. In three spatial dimensions, we add the constraints
\begin{gather}\label{eq:define-s-edge}
    \langle \liftingOp_\level \mu, \eta \rangle_\Gamma =  \langle \avg{\mu}_\Gamma, \eta \rangle_\Gamma 
    \qquad \forall \eta \in \polynomials_{p+2}(\Gamma),
\end{gather}
\end{subequations}
Here, $\avg{m}_\Gamma$ is the average taken over all cells adjacent to $\Gamma$.
This construction extends to dimensions higher than three. For a precise definition of $V_\textup{disc}$ and a proof of the fact that $S_\level$ is well-defined, please refer to \cite[Lem.\ A.3]{GiraultR1986} and \cite[Def.\ 5.46]{Monk2003}.
Moreover, we have
\begin{lemma}[Properties of $\liftingOp_\level \lambda$]
 Under assumptions \eqref{EQ:LS1}--\eqref{EQ:LS4}, we have
 \begin{align}
  \| \liftingOp_\level \mu \|_0 ~ \cong ~ & \| \mu \|_\level && \forall \lambda \in \skeletalSpace_\level && \text{(norm equiv.)}\label{EQ:lift_equiv}\\
  \liftingOp_\level \gamma_\level w ~=~ & w && \forall w \in \linElementSpace_\level && \text{(lifting identity)}\label{EQ:lift_iden}\\
  |\liftingOp_\level \mu|_{1} ~\lesssim~ & \|  \mu \|_{a_\level} && \forall \lambda \in \skeletalSpace_\level &&  \text{(lifting bound)}\label{EQ:lift_bound}
 \end{align}
\end{lemma}
\begin{proof}
This is \cite[Lem.\ 5.5]{LuRK22a} with our modified version of \eqref{EQ:LS1} as compared to their (LS1).
\end{proof}
\begin{proof}[Proof of Lemma \ref{TH:conv_result}]
 Inequalities \eqref{EQ:Q_mehses} and \eqref{EQ:second_ineq} can be obtained as in \cite[Lem.\ 5.1]{LuRK22a}. Let us obtain inequality \eqref{EQ:i_stable}. To this end, we observe that
 \begin{equation*}
  \| \injectionOp_\level \mu \|_{a_{\level}}^2 \overset{\eqref{EQ:LS9}}{\lesssim} \| \localQ_\level \injectionOp_\level \mu \|^2_\Omega + \sum_{\elem \in \mesh_\level} \tfrac1{h_\elem} \| \localU_\level \injectionOp_\level \mu - \localU_{\level - 1} \mu \|^2_{\partial \elem} + \sum_{\elem \in \mesh_\level} \tfrac1{h_\elem} \| \localU_{\level - 1} \mu - \injectionOp_\level \mu \|^2_{\partial \elem},
 \end{equation*}
 where \eqref{EQ:Q_mehses} allows us to bound $\| \localQ_\level \injectionOp_\level \mu \|^2_\Omega$, \eqref{EQ:second_ineq} allows us to bound
 \begin{equation*}
  \sum_{\elem \in \mesh_\level} h^{-1}_\elem \| \localU_\level \injectionOp_\level \mu - \localU_{\level - 1} \mu \|^2_{\partial \elem} \lesssim h^{-2}_\level \| \localU_\level \injectionOp_\level \mu - \localU_{\level - 1} \mu \|^2_{\Omega}.
 \end{equation*}
 Last, we consider the reference element $\hat \elem$, which is the center of a reference element patch $\omega_{\hat \elem}$ with skeleton $\omega_{\faceSet}$, and the face-wise $L^2$ orthogonal projection $\pi \colon L^2(\omega_{\faceSet}) \to \skeletalSpace_{\omega_{\faceSet}}$. We observe that
 \begin{equation*}
  \mathcal G\colon H^1(\omega_{\hat \elem}) \ni v \mapsto \localU_{\level - 1} \pi v - \injectionOp_\level \pi v \in \skeletalSpace_\level|_{\partial \hat \elem} \subset L^2(\partial \hat \elem)
 \end{equation*}
 is a continuous operator, which vanishes if $v$ is an overall continuous, element-wise linear polynomial. Thus, the standard scaling argument allows us to deduce that
 \begin{equation*}
  \sum_{\elem \in \mesh_\level} \tfrac1{h_\elem} \| \localU_{\level - 1} \pi v - \injectionOp_\level \pi v \|^2_{\partial \elem} \lesssim | v |^2_{H^1(\Omega)}.
 \end{equation*}
 Finally, setting $v = \liftingOp_{\level-1} \mu$ with the lifting operator $\liftingOp_{\level -1}$ as defined in \eqref{eq:define-s} yields the result (in conjunction with \eqref{EQ:lift_bound}).
 
 Relation \eqref{EQ:ap_stable} results from
 \begin{multline*}
%   \label{EQ:stability-P}
   \| \projectionOp_{\level-1} \mu \|^2_{a_{\level-1}}
   = a_{\level-1}(\projectionOp_{\level-1} \mu, \projectionOp_{\level-1} \mu)
     \overset{\eqref{EQ:projection_definition}}= a_{\level}(\mu, \injectionOp_{\level}\projectionOp_{\level-1} \mu)
   \\
   \le \| \mu \|_{a_\level} \| \injectionOp_\level \projectionOp_{\level-1} \mu \|_{a_\level}
   \overset{\eqref{EQ:i_stable}}\lesssim \| \mu \|_{a_\level} \| \projectionOp_{\level-1} \mu \|_{a_{\level-1}},
 \end{multline*}
 where we have used the Bunyakovsky--Cauchy--Schwarz inequality to pass to the second line.
\end{proof}
\begin{lemma}
 Under the assumptions of Lemma \ref{TH:conv_result}, \eqref{EQ:precond2} holds.
\end{lemma}
\begin{proof} Using first the linearity and symmetry of $a_{\level}$ and then the definition \eqref{EQ:projection_definition} of the Ritz quasi-projector $\projectionOp_{\level-1}$, we obtain
  \[
 \begin{aligned}
   &a_\level ( \mu -  \injectionOp_{\level} \projectionOp_{\level-1} \mu, \mu - \injectionOp_\level \projectionOp_{\level-1} \mu )
   \\
   &\quad=
   a_\level ( \mu, \mu ) - 2 a_\level ( \mu, \injectionOp_\level \projectionOp_{\level-1} \mu )
   + a_\level ( \injectionOp_\level \projectionOp_{\level-1} \mu, \injectionOp_\level \projectionOp_{\level-1} \mu ) \notag\\
   &\quad
   \le \| \mu \|_{a_\level}^2 \underbrace{ - 2 a_{\level - 1} ( \projectionOp_{\level-1} \mu,  \projectionOp_{\level-1} \mu ) }_{ \le 0 }
   + \| \injectionOp_\level \projectionOp_{\level-1} \mu \|_{a_\level}^2.
 \end{aligned}
 \]
   To conclude, we write for the rightmost term
   \[
   \| \injectionOp_\level \projectionOp_{\level-1} \mu \|_{a_{\level}}^2
   \overset{\eqref{EQ:i_stable}}\lesssim \| \projectionOp_{\level-1} \mu \|_{a_{\level-1}}^2
   \overset{\eqref{EQ:ap_stable}}\lesssim \| \mu \|^2_{a_\level}.\qedhere
   \]
\end{proof}
%
% ---------------------------------------------------------------------
\subsection{Proofs of \eqref{EQ:precond1a} and \eqref{EQ:precond1}}
% ---------------------------------------------------------------------
% 
A classical approach for multigrid proofs consists is considering $A_\level \mu$ (cf. \eqref{eq:Al}) as a function in $L^2(\Omega)$ and using it as right-hand side in an auxiliary problem. However, this is not possible for discontinuous skeletal methods, since $A_\level \mu \in \skeletalSpace_\level$ is defined only on the mesh skeleton and not on the whole domain $\Omega$. To this end, we use the lifting operator $\liftingOp_\level$. That is, we define the auxiliary right-hand side $f_\mu \in V_\textup{disc}$ as the unique solution of
\begin{equation}\label{EQ:scp_def}
   (f_\mu, \liftingOp_\level \eta)
   = \langle A_\level \mu, \eta \rangle_{\level}
    = a_\level (\mu, \eta)
   \qquad \forall \eta \in \skeletalSpace_\level,
\end{equation}
and its perturbed skeletal approximation $\tilde \mu \in \skeletalSpace_\level$ with
\begin{equation}\label{EQ:tilde_lambda}
 a_\level (\tilde \mu, \eta) = (f_\mu, \localU_\level \eta) \qquad \forall \eta \in \skeletalSpace_\level.
\end{equation}
Notably, the solutions of this discrete, perturbed problem approximate the solution of the continuous problem
\begin{equation}\label{EQ:Pois_app}
   (\nabla\tilde u,\nabla v) = (f_\mu,v)
 \qquad \forall v\in H^1_0(\Omega).
\end{equation}
\begin{lemma}\label{TH:extract_h_basis}
  Let \eqref{EQ:LS1}--\eqref{EQ:LS4} hold.
  Then we have, for all $\mu \in \skeletalSpace_\level$,
 \begin{equation}\label{EQ:bound_flambda}
  \| \mu - \tilde \mu \|_{a_\level} \lesssim h_\level \| f_\mu \|_\Omega, \qquad \text{ and } \qquad
  \| f_\mu \|_\Omega \lesssim \| A_\level \mu \|_\level.
 \end{equation}
\end{lemma}
\begin{proof}
 This is part of \cite[Lem.\ 5.6]{LuRK22a}. The change in the meaning of \eqref{EQ:LS1} does not influence this result.
\end{proof}
\begin{lemma}[Reconstruction approximation]\label{LEM:reconstruction_approx}
  Let \eqref{EQ:IA1}, \eqref{EQ:IA2}, and \eqref{EQ:LS1}--\eqref{EQ:LS5} hold.
  Then, if the model problem has regularity \eqref{EQ:regularity}, for all $\mu \in \skeletalSpace_\level$ there exists an auxiliary function $\ureconstructed \in \linElementSpace_{\level - 1}$ such that
 \begin{equation}\label{EQ:reconst_approx}
  \| \localQ_\level \mu + \nabla \ureconstructed \|_\Omega + \| \localQ_{\level - 1} \projectionOp_{\level - 1} \mu + \nabla \ureconstructed \|_\Omega \lesssim h_\level^{\alpha} \| f_\mu \|_{\alpha - 1}.
 \end{equation}
\end{lemma}
\begin{proof}
 This is \cite[Lem.\ 5.8]{LuRK22a}.
\end{proof}
\begin{lemma}\label{LEM:approxA}
Let \eqref{EQ:LS1}--\eqref{EQ:LS4}, \eqref{EQ:LS6}, and \eqref{EQ:LS7} hold. Then,
\begin{gather*}
    \|f_\mu\|_{-1} \lesssim \| \mu \|_{a_\level}.
\end{gather*}
\end{lemma}
\begin{proof}
  By the definition of negative norms and properties of $\sup$,
 \begin{equation}\label{eq:approxA:basic}
  \|f_\mu\|_{-1} %= \sup_{\psi \in H^1_0(\Omega)} \frac{(f_m, \psi)}{| \psi |_1}
  \le  \sup_{\psi \in H^1_0(\Omega)} \frac{(f_\mu, \psi - \liftingOp_\level \traceOp \contLinProj_\level \psi)}{| \psi |_1} + \sup_{\psi \in H^1_0(\Omega)} \frac{(f_\mu, \liftingOp_\level \traceOp \contLinProj_\level \psi)}{| \psi |_1}.
 \end{equation}
 Here, $\contLinProj_\level$  is an (quasi-)interpolator onto $\linElementSpace_{\level}$ that satisfies
 \begin{align}
  | \contLinProj_{\level} v |_1 & \lesssim | v |_1 && \forall v \in H^1(\Omega), \label{EQ:H1_stab}\\
  \| v - \contLinProj_{\level} v \|_\Omega & \lesssim h_{\level}^{1-k+\alpha} |v|_{\alpha+1},  && \forall v \in H^{\alpha+1}(\Omega), \; k=0,1.\label{EQ:H1_approx}   
 \end{align}
 An example is given in \cite{SZ90}. We continue with bounding
 \begin{equation}\label{eq:approxA:1}
   \| \psi - \liftingOp_\level \traceOp \contLinProj_\level \psi \|_\Omega
   \overset{\eqref{EQ:lift_iden}}= \| \psi - \contLinProj_\level \psi \|_\Omega \overset{\eqref{EQ:H1_stab}}\lesssim h_\level | \psi |_1,
 \end{equation}
 and observe that
 \begin{equation*}
   (f_\mu, \liftingOp_\level \traceOp \contLinProj_\level \psi)_\Omega  
   \overset{\eqref{EQ:scp_def}}=
   a_\level (\mu, \gamma_\level \contLinProj_\level \psi),
 \end{equation*}
 which immediately yields, using a Cauchy--Schwarz inequality,
 \begin{equation}\label{eq:approxA:2}
   | (f_\mu, \liftingOp_\level \traceOp \contLinProj_\level \psi)_\Omega |
   \le \| \mu \|_{a_\level} \| \gamma_\level \contLinProj_\level \psi \|_{a_\level} = \| \mu \|_{a_\level} | \contLinProj_\level \psi |_1 \lesssim \| \mu \|_{a_\level} | \psi |_1,
 \end{equation}
 where $|\cdot|_1$ denotes the $H^1(\Omega)$-seminorm. Applying a Cauchy--Schwarz inequality to the numerator of the first supremum in \eqref{eq:approxA:basic} followed by \eqref{eq:approxA:1} and using \eqref{eq:approxA:2} to estimate the numerator of the second supremum, we get, after simplification,
 \begin{equation*}
   \|f_\mu\|_{-1} \le h_\level \| f_\mu \|_\Omega + \| \mu \|_{a_\level}
   \overset{\eqref{EQ:bound_flambda}}\lesssim
   h_\level \| A_\level \mu \|_\level + \| \mu \|_{a_\level} \lesssim \| \mu \|_{a_\level},
 \end{equation*}
 where the last inequality is the rightmost inequality of \eqref{EQ:LS6}.
\end{proof}

With these preliminary results at hand, we can formulate a generalization of \cite[Theo.~4.1]{LuRK21} and \cite[Theo.\ 5.10]{LuRK22a}:
\begin{theorem}\label{TH:A1_proof}
  If \eqref{EQ:regularity} holds with $\alpha\in (1/2,1]$, and if the assumptions \eqref{EQ:IA1}, \eqref{EQ:IA2}, and \eqref{EQ:LS1}--\eqref{EQ:LS9} hold, then \eqref{EQ:precond1a} is satisfied.
\end{theorem}
\begin{remark}
 Theorem \ref{TH:A1_proof} implies \eqref{EQ:precond1} if $\alpha = 1$.
\end{remark}
\begin{proof}
  We use the definitions \eqref{EQ:bilinear} of $a_\level$ and \eqref{EQ:projection_definition} of $\projectionOp_{\level}$ to obtain
  \begin{equation}\label{eq:A1:basic}
    \begin{aligned}
      a_\level(\mu  - \injectionOp_\level \projectionOp_{\level-1} \mu, \mu)
      &= a_\level(\mu, \mu) - a_{\level-1}(\projectionOp_{\level-1} \mu, \projectionOp_{\level-1} \mu) \notag\\
      &=
      \underbrace{%
        (\localQ_\level \mu, \localQ_\level \mu)_\Omega - (\localQ_{\level-1} \projectionOp_{\level-1} \mu, \localQ_{\level-1} \projectionOp_{\level-1} \mu)_\Omega
      }_{\mathfrak{T}_1}
      \\
      &\quad + \underbrace{%
        s_{\level}(\mu, \mu) - s_{\level - 1}(\projectionOp_{\level -1} \mu, \projectionOp_{\level -1} \mu),
      }_{\mathfrak{T}_2}
    \end{aligned}
  \end{equation}
    and analyze the respective contributions $\mathfrak{T}_1$ and $\mathfrak{T}_2$ separately.
 
 First, binomial factorization yields
 \begin{gather*}
   \mathfrak{T}_1
  =  (\localQ_\level \mu + \localQ_{\level-1} \projectionOp_{\level-1} \mu, \localQ_\level \mu - \localQ_{\level-1} \projectionOp_{\level-1} \mu)_\Omega.
 \end{gather*}
 Let now $w \in \linElementSpace_{\level-1}$. Invoking the orthogonality property stated in Lemma \ref{LEM:quasi_orth} to insert $2\nabla w$ into the first argument of the $L^2$-product, and adding and subtracting $\nabla w$ to the second argument, we obtain
 \begin{equation}\label{EQ:T1_bound}
   \begin{aligned}
     \mathfrak{T}_1
     &=  (\localQ_\level \mu + 2 \nabla w + \localQ_{\level-1} \projectionOp_{\level-1} \mu, \localQ_\level \mu + \nabla w - \nabla w - \localQ_{\level-1} \projectionOp_{\level-1} \mu)_\Omega
     \\
     &\le
     \left(
       \| \localQ_\level + \nabla w \|_\Omega + \| \localQ_{\level-1} \projectionOp_{\level-1} \mu + \nabla w \|_{\Omega}
       \right)^2
       \overset{\eqref{EQ:reconst_approx}}\lesssim h_\level^{2\alpha} \| f_\mu \|_{\alpha - 1}^2,
   \end{aligned}
 \end{equation}
 where we have used Bunyakovsky--Cauchy--Schwarz and triangle inequalities to pass to the second line.

 Second, we obtain a similar estimate for $\mathfrak{T}_2$. To this end, we detail the treatment of the first summand of $\mathfrak{T}_2$, the second summand can be treated analogously.
 For any $w\in \linElementSpace_{\level}$, it holds
 \begin{equation*}
  s_\level (\mu, \mu)
  \overset{\eqref{EQ:LS7}}= s_\level(\mu - \gamma_\level w, \mu - \gamma_\level w)
  \le a_\level(\mu -  \gamma_\level w, \mu - \gamma_\level w) = \|\mu - \gamma_\level w\|^2_{a_\level}.
 \end{equation*}
 Using the triangle inequality and Lemma \ref{TH:extract_h_basis}, we have that
 \begin{equation*}
  \|\mu - \gamma_\level w\|_{a_\level} \le \|\mu - \tilde \mu\|_{a_\level} + \|\tilde \mu - \gamma_\level w\|_{a_\level} \le h_\level \| f_\mu \|_\Omega + \|\tilde \mu - \gamma_\level w\|_{a_\level}
 \end{equation*}
 Defining $w$ as the $L^2$ orthogonal projection of $\tilde u$ into $\linElementSpace_\level$ yields the result with \eqref{EQ:LS5} as we can write
 \begin{equation*}
  \|\tilde \mu - \gamma_\level w\|_{a_\level} \le \|\tilde \mu - \skeletalProj_\level \tilde u \|_{a_\level} + \| \skeletalProj_\level \tilde u - \gamma_\level w\|_{a_\level},
 \end{equation*}
 where the first summand is bounded via \eqref{EQ:LS5} and the second summand can be bounded using \eqref{EQ:LS6} and standard approximation properties.
 
 Proceeding similarly for the second summand in $\mathfrak{T}_2$, we infer
 \begin{equation}
  \mathfrak{T}_2 \lesssim h_\level^{2\alpha} \| f_\mu \|_{\alpha - 1}^2.\label{EQ:T2_bound}
 \end{equation}

 Using \eqref{EQ:T1_bound} and \eqref{EQ:T2_bound} to estimate the right-hand side of \eqref{eq:A1:basic} and invoking Sobolev interpolation provides us with
 \begin{equation*}
  | a_\level (\mu - \injectionOp_\level \projectionOp_{\level - 1} \mu, \mu) | \lesssim h_\level^{2\alpha} \| f_\mu \|_{\alpha-1}^2 \lesssim h_\level^{2\alpha} \| f_\mu \|^{2(1-\alpha)}_{-1} \| f_\mu \|_\Omega^{2\alpha},
 \end{equation*}
 and Lemmas~\ref{TH:extract_h_basis} and~\ref{LEM:approxA} yield
 \begin{equation*}
  | a_\level (\mu - \injectionOp_\level \projectionOp_{\level - 1} \mu, \mu) | \lesssim  h_\level^{2\alpha} \| A_\level \mu \|^{2\alpha}_\level\| \mu \|^{2(1-\alpha)}_{a_\level}.
 \end{equation*}
 This is the result, since \eqref{EQ:LS6} implies that $\underline \lambda^A_\level \lesssim h_\level^{-2}$.
\end{proof}

\section{Verification of assumptions for HHO} \label{SEC:hho_proofs}
\subsection{Preliminaries}

For all $T \in \mesh_\level$ (resp. $F \in \skeleton_\level$), we denote by $\| \cdot \|_T$ (resp. $\| \cdot \|_F$) the standard $L^2$-norm over $T$ (resp. $F$).
For all $T\in\mesh_\level$, we introduce the set $\faceSet_\elem$ collecting the faces of $\elem$.

\subsubsection{General properties of used norms and function spaces}
Let $\elem\in \mesh_\level$, $\face\in \faceSet_\elem$.
\begin{itemize}
    \item The \emph{discrete trace inequality} (see, e.g., \cite[Lem.\ 1.32]{di_pietro_hybrid_2020}), stipulates that 
    \begin{equation} \label{EQ:discr_trace_inequality}
        \|v\|_\face \lesssim h_\elem^{-1/2} \|v\|_\elem \qquad \forall v \in \polynomials(\elem).
    \end{equation}
    \item The \emph{discrete inverse inequality} (see, e.g., \cite[Lem.\ 1.28]{di_pietro_hybrid_2020}) stipulates that 
    \begin{equation} \label{EQ:inverse_inequality} 
        \|\nabla v\|_{T} \lesssim h_\elem^{-1} \|v\|_\elem \qquad \forall v \in \polynomials(\elem).
    \end{equation}
  \item The \emph{Poincar\'{e}-Friedrichs inequality} (see, e.g., \cite[Lem.\ 3.30]{ErnG21v1}) stipulates that 
    \begin{equation} \label{EQ:General_PF_inequality} 
        \|\rho \|_{T} \lesssim h_\elem \|\nabla \rho\|_\elem  + h_\elem^{1/2} \| \rho\|_{\partial \elem} \qquad \forall \rho \in H^1(\elem). 
    \end{equation}
  \item Finally, we have the useful relation
    \begin{equation} \label{EQ:norm_l_rewrite}
      \|\mu\|_\level^2 \simeq
      \sum_{F\in \skeleton_\level} h_\face  \|\mu\|_F^2 \qquad \forall \mu \in \skeletalSpace_\level
    \end{equation}
    with $\faceSet_\level$ containing all the faces of $\mesh_\level$. Thus, the square root of the expression in the right-hand side induces a global norm on $\skeletalSpace_\level$.
\end{itemize}

\subsubsection{Properties of HHO}
In this part, we enlist some properties of HHO that will be used to show \eqref{EQ:LS1}--\eqref{EQ:LS9}.

Recalling the hybrid bilinear form $\underline{a}_\level$ introduced in \eqref{EQ:hho_bilinear_uncondensed}, we consider the hybrid, discrete problem: find $(u_\level, m_\level)\in \discElementSpace_\level\times\skeletalSpace_\level$ such that
\begin{equation} \label{EQ:hybrid_pb}
    \underline{a}_\level((u_\level, m_\level), (v, \mu)) = \int_\Omega fv\dx
    \qquad \forall (v, \mu)\in \discElementSpace_\level\times\skeletalSpace_\level.
\end{equation}
The pair $\underline{u}_\level \coloneq (\localU_\level m_\level + \localV_\level f, m_\level) \in \discElementSpace_\level\times\skeletalSpace_\level$ is solution of \eqref{EQ:hybrid_pb} if and only if $m_\level \in \skeletalSpace_\level$ is solution of the condensed problem \eqref{EQ:condensed_pb} (see \cite[Prop.\ 4]{cockburn_bridging_2016}). 
In particular, $\localU_\level m_\level + \localV_\level f$ approximates the exact solution $u$ of the continuous problem \eqref{EQ:poisson_pb}.

As in \cite[2.1.2 (2.7), 2.2.2 (2.35)]{di_pietro_hybrid_2020}, we define
\begin{subequations}\label{EQ:skel_norms}
\begin{gather}
 |(v, \mu)|_{\underline{1,\level}}^2 \coloneq \sum_{\elem \in \mesh_\level}  |(v, \mu) |_{\underline{1,T}}^2 = \sum_{\elem \in \mesh_\level} \sum_{\face\in \faceSet_\elem} h_\face^{-1} \|v - \mu\|_F^2, \\
 \|(v, \mu)\|_{\underline{1,\level}}^2
 \coloneq \sum_{\elem \in \mesh_\level} \|(v, \mu)\|_{\underline{1,T}}^2
 = \sum_{\elem \in \mesh_\level}  \left(
 \|\nabla v\|_\elem^2 +  |(v, \mu)|_{\underline{1,T}}^2
 \right), \label{eq:norm.u1l}
\end{gather}
\end{subequations}
as the $H^1$-like seminorm on the hybrid space $\discElementSpace_\level\times\skeletalSpace_\level$. Moreover, we assume that the properties of \cite[Assumption\ 2.4]{di_pietro_hybrid_2020} for the stabilization bilinear form $\underline{s}_\elem$ hold true. Then, we have:
\begin{itemize}
 \item (Boundedness and stability of $\underline{a}_\elem$ \cite[Lem.\ 2.6]{HHObookCEP2022})
    \begin{equation} \label{EQ:hho_hybrid_local_stability}
        \|(v, \mu)\|_{\underline{1,\elem}}^2 \lesssim \underline{a}_\elem((v, \mu),(v, \mu)) \lesssim \|(v, \mu)\|_{\underline{1,\elem}}^2.
    \end{equation}
    \item (Energy error estimate \cite[Lem.\ 2.8, Lem.\ 2.9]{HHObookCEP2022})
    \begin{equation} \label{EQ:hho_energy_error}
        \| \vec q \underline{u}_\level - \vec q( \discProj_\level u, \skeletalProj_\level u) \|_{\underline{a}_\level} 
        %\lesssim h_\level |u|_{H^2}.
        \lesssim h_\level^{\alpha} \| u\|_{\alpha + 1},
    \end{equation}
    where $\alpha>\frac{1}{2}$ and $\| \cdot \|_{\underline{a}_\level}$ is the norm induced by the bilinear form $\underline{a}_\level$.
    \item (Cell unknown $L^2$-error estimate  \cite[Lem.\ 2.11]{HHObookCEP2022})\newline
    If $p \geq 1$, $\alpha>\frac{1}{2}$,  then
    \begin{equation} \label{EQ:hho_cell_ukn_error}
        \| \localU_\level m_\level + \localV_\level f - \discProj_\level u \|_\Omega \lesssim h_\level^{\alpha+\delta} \| u\|_{\alpha + 1},
    \end{equation}
where $\delta\coloneq\min\{\alpha, 1\}$.
    \item (Face unknown $L^2$-error estimate  \cite[Lem.\ 2.9]{HHObookCEP2022})\newline
    If $p \geq 1$,  then
    \begin{equation} \label{EQ:hho_face_ukn_error}
        \| m_\level - \skeletalProj_\level u \|_\level \lesssim h_\level^{\alpha+1} \| u\|_{\alpha + 1}.
    \end{equation}
\end{itemize}
Now, let us introduce the seminorms $\|\cdot\|_{1,\level}$ and $|\cdot|_{1,\level}$ on the skeletal space $\skeletalSpace_\level$, which are based on the seminorms $\|\cdot\|_{\underline{1,\level}}$ and $|\cdot|_{\underline{1,\level}}$ defined in \eqref{EQ:skel_norms}:
\begin{equation} \label{EQ:hho_norm}
 \|\mu\|_{1,\level} \coloneq \|(\localU_\level \mu, \mu)\|_{\underline{1,\level}} \qquad \text{ and } \qquad |\mu|_{1,\level} \coloneq |(\localU_\level \mu, \mu)|_{\underline{1,\level}}.
\end{equation}
Similarly, the \emph{condensed} bilinear form $a_\level$ introduced in \eqref{EQ:hho_bilinear} is built from the \emph{hybrid} bilinear form $\underline{a}_\level$ \eqref{EQ:hho_bilinear_uncondensed}, in which the generic cell unknown variable is also recovered from the skeletal variable through $\localU_\level$, i.e.
\begin{equation*}
    a_\level(m, \mu) = \underline{a}_\level((\localU_\level m, m), (\localU_\level\mu, \mu)).
\end{equation*}
In the same fashion, we also have
\begin{equation*}
    s_\level(m, \mu) = \underline{s}_\level((\localU_\level m, m), (\localU_\level\mu, \mu)).
\end{equation*}
Thus, useful properties of $a_\level$ and $s_\level$ derive in a natural way from those of $\underline{a}_\level$ and $\underline{s}_\level$. In particular, we shall use in this work the following results:
\begin{itemize}
    \item (Boundedness and stability of $a_\level$)
    \begin{equation} \label{EQ:hho_stability}
        \|\mu\|_{1,\level}^2 \lesssim a_\level(\mu, \mu) \lesssim \|\mu\|_{1,\level}^2.
    \end{equation}
    \item Since $\underline{s}_\level$ is symmetric positive semi-definite \cite[Assumption 2.4]{di_pietro_hybrid_2020}, so is $s_\level$.
\end{itemize}

\subsubsection{The norm $\|\cdot\|_\level$ is weaker than $\|\cdot\|_{1,\level}$}
In this section, we prove that
\begin{equation} \label{EQ:norm_equiv1}
    \|\mu\|_\level \lesssim \|\mu\|_{1,\level}  %\lesssim h_\level^{-1}\|\mu\|_\level
\end{equation}
%Let us now prove the reverse inequality. 
Starting from \eqref{EQ:norm_l_rewrite}, the triangle inequality is used via the insertion of $0 = \localU\mu - \localU\mu$ into the norms in the sum, yielding
\begin{equation}\label{EQ:LS6:proof3}
  \begin{aligned}
    \|\mu\|^2_\level
    &\lesssim \sum_{F\in \skeleton} h_\face  \left( \|\localU_\level\mu - \mu\|^2_F + \|\localU_\level\mu\|^2_F \right)
    \\
    \overset{\eqref{EQ:discr_trace_inequality}}&\lesssim
    \sum_{F\in \skeleton} h_\face  \|\localU_\level\mu - \mu\|^2_F + \sum_{T\in\mesh_\level}  \|\localU_\level\mu\|^2_T
    \\
    &\lesssim
    \sum_{F\in \skeleton} h_\face  \|\localU_\level\mu - \mu\|^2_F
    + \|(\localU_\level \mu, \mu)\|_{\underline{1,\level}}^2,
  \end{aligned}
\end{equation}
where the last passage follows from the discrete Poincar\'e inequality {\cite[Lem.\ 2.15]{di_pietro_hybrid_2020}} applied to $(\localU_\level \mu, \mu)$ (notice that $\mu$ vanishes on $\partial \Omega$).
Hence, recalling \eqref{EQ:hho_norm},
%% Moreover, the discrete Poincar\'e inequality \cite[Rem.\ 4.15]{PietroErn} states that
%% \begin{equation} \label{EQ:LS6:proof4}
%%     \sum_{T\in\mesh_\level} \|\localU_\level\mu\|_T^2 = \|\localU_\level\mu\|_\Omega^2 \lesssim \|\nabla_\level\localU_\level\mu\|_\Omega^2 + \sum_{F\in\skeleton} h_F^{-1} \|\jump{\localU_\level\mu}_F\|_F^2.
%% \end{equation}
%% Rewriting $\jump{\localU\mu}_F$ as
%% \begin{equation*}
%%     \jump{\localU_\level\mu}_F =
%%     \begin{cases}
%%     (\localU_\level\mu - \mu)_{|F} & \text{ if } F \subset \partial\Omega, \\
%%     (\localU_\level\mu_{|T_1} - \mu)_{|F} + (\mu - \localU_\level\mu_{|T_2})_{|F} & \text{ if } F \in \faceSet_{T_1}\cap \faceSet_{T_2},
%%     \end{cases}
%% \end{equation*}
%% then, through the triangle inequality we infer that
%% \begin{align} 
%%     \sum_{F\in\skeleton} h_F^{-1} \|\jump{\localU_\level\mu}_F\|^2_F 
%%     \leq 4 \sum_{T\in\mesh_\level} \sum_{F\in \faceSet_T} h_F^{-1} \|\localU_\level\mu - \mu\|_F^2 %\nonumber
%%     =: 4 \, \mathfrak{T}(\mu). \label{EQ:LS6:proof5}
%% \end{align}
%% Plugging \eqref{EQ:LS6:proof5} into \eqref{EQ:LS6:proof4}, and then into \eqref{EQ:LS6:proof3} gives
\begin{equation} \label{EQ:norm_equiv1:proof4}
    \|\mu\|_\level^2 
    \lesssim 
    \sum_{F\in \skeleton} h_\face  \|\localU_\level\mu - \mu\|_F^2 + \|\nabla_\level\localU_\level\mu\|_\Omega^2 +|\mu|_{1,\level}
\end{equation}
%% Finally, the first term is eliminated the following way. It holds that
%% \begin{equation*}
%%     \sum_{F\in \skeleton} h_\face  \|\localU_\level\mu - \mu\|_F^2
%%     =
%%     h_\face^2 \sum_{F\in \skeleton} h_\level^{-1}  \|\localU_\level\mu - \mu\|_F^2
%%     \corr{\approx}{\simeq}{[DDP]} 
%%     h_\face^2 \, \mathfrak{T}(\mu).
%% \end{equation*}
%% Considering small values of $h_\face$, $h_\face^2 \, \mathfrak{T}(\mu)$ can be neglected before $\mathfrak{T}(\mu)$, and \eqref{EQ:norm_equiv1:proof4} becomes
The fact that $h_F \lesssim h_F^{-1}$ is a consequence of the fact that we assumed that $\Omega$ has diameter 1 without loss of generality so that $h_F \lesssim 1$.
\begin{equation*}
    \|\mu\|_\level^2 \lesssim \|\nabla_\level\localU_\level\mu\|_\Omega^2 + |\mu|_{1,\level} = \|\mu\|_{1,\level}^2,
\end{equation*}
and the result follows by taking the square root.

\subsection{Verification of \eqref{EQ:LS1}}
We write
\[
\begin{aligned}
      \| \localU_\level \mu - \mu \|_\level^2
      &\lesssim \sum_{T \in \mathcal{T}_h} \sum_{F \in \mathcal{F}_T} h_F \| \localU_\level \mu - \mu \|_{L^2(F)}^2
      \\
      &\le h_\level^2 \sum_{T \in \mathcal{T}_h} \sum_{F \in \mathcal{F}_T} h_F^{-1} \| \localU_\level \mu - \mu \|_{L^2(F)}^2
      \\
      \overset{\eqref{eq:norm.u1l}}&\le h_\level^2 \| (\localU_\level \mu, \mu) \|_{\underline{1,h}}^2
      \overset{\eqref{EQ:hho_hybrid_local_stability}}\lesssim h_\level^2 \| (\localU_\level \mu, \mu) \|_{\underline{a}_\level}^2
      \overset{\eqref{eq:norm.al}}= h_\level^2 \| \mu \|_{a_\level}^2,
    \end{aligned}
    \]
    where the first inequality follows from mesh regularity.
    Taking the square root of the above inequality proves \eqref{EQ:LS1}.

% Verifying \eqref{EQ:LS1} is the most challenging part for HHO,and we do not prove it here for all available HHO versions.
% We consider here the stabilization bilinear form
% % 
% \begin{multline} \label{EQ:hho_stabilizer}
%  s_\level(m, \mu) = \sum_{\elem \in \mesh_\level} \int_{\partial \elem} \Pi^\partial_\level (\localU_\level m - m + \theta_\level^{p+1} (\localU_\level m, m) - \Pi^\textup d_\level \theta_\level^{p+1} (\localU_\level m, m)) \\ \Pi^\partial_\level (\localU_\level \mu - \mu + \theta_\level^{p+1} (\localU_\level \mu, \mu) - \Pi^\textup d_\level \theta_\level^{p+1} (\localU_\level \mu, \mu)) \ds.
% \end{multline}
% This stabilization term builds the standard HHO method \cite[Example 2.7]{di_pietro_hybrid_2020}, for which \eqref{EQ:LS1} is shown in \cite[Lem.\ 39.2]{ErnG21b}. We leave it to the reader to show \eqref{EQ:LS1} for their favorite HHO variant.

\subsection{Verification of \eqref{EQ:LS2}}

Let us start with the second inequality in \eqref{EQ:LS2} on an arbitrary element $\elem \in \mesh_\level$. Recalling the definition \eqref{EQ:hho_localU} of $\localU_\ell$ and choosing $v_T = u_T^1 =  \localU_\elem \mu $ and $ m_{\partial \elem} =\mu $, we have, denoting by $\| \cdot \|_{\underline{a}_\elem}$ the seminorm induced by $\underline{a}_\elem$,
\begin{multline*}
 \| (\localU_\elem\mu, 0) \|_{\underline{a}_\elem}^2 =
  \underline{a}_\elem((\localU_\elem\mu, 0), (\localU_\elem\mu, 0))
  \\
  = - \underline{a}_\elem((0, \mu), (\localU_\elem\mu, 0))
    \le \| (0,\mu) \|_{\underline{a}_\elem}\,\| (\localU_\elem\mu, 0) \|_{\underline{a}_\elem},
\end{multline*}
  where the conclusion follows from the Cauchy--Schwarz inequality. Hence,
  \begin{equation}\label{eq:norm.Umu.le.mu}
    \| (\localU_\elem\mu, 0) \|_{\underline{a}_\elem}
    \le \| (0,\mu) \|_{\underline{a}_\elem}
  \end{equation}
  We write
\begin{equation}\label{eq:norm.U.1T.le.norm.mu.1T}
  \|(\localU_\elem\mu, 0)\|_{\underline{1,\elem}}^2
   %% \lesssim \underline{a}_\elem((\localU_\elem\mu, 0), (\localU_\elem\mu, 0))
  \overset{\eqref{EQ:hho_hybrid_local_stability}}\lesssim \| (\localU_\elem\mu, 0) \|_{\underline{a}_\elem}^2
  \overset{\eqref{eq:norm.Umu.le.mu}}\leq  %% \underline{a}_\elem((0, \mu), (0, \mu))
  \| (0,\mu) \|_{\underline{a}_\elem}^2
  \overset{\eqref{EQ:hho_hybrid_local_stability}}\lesssim \|(0, \mu)\|_{\underline{1,\elem}}^2
   \lesssim h_\elem^{-1} \| \mu \|_{\partial T}^2,
\end{equation}
  where the last inequality follows recalling the definition \eqref{eq:norm.u1l} of $\| \cdot \|_{\underline{1,\elem}}$ and noticing that $h_F^{-1} \lesssim h_T^{-1}$ for all $F \in \faceSet_\elem$ by mesh regularity.
Next, using the Poincar\'{e}-Friedrichs inequality \eqref{EQ:General_PF_inequality} and definition \eqref{eq:norm.u1l} of the $  \|(\cdot, \cdot)\|_{\underline{1,\elem}}$-norm along with the fact that $h_\elem^{-1} \lesssim h_\face^{-1}$ for all $\face \in \faceSet_\elem$ by mesh regularity,   we have 
\begin{equation}\label{Proof:LS2 second}
 \| \localU_\elem \mu \|_\elem^2 
 \lesssim  h_\elem^2 \| \nabla (\localU_\elem \mu) \|_\elem^2  + h_\elem \| \localU_\elem \mu \|_{\partial \elem}^2
 \lesssim h_T^2  \|(\localU_\elem\mu, 0)\|_{\underline{1,\elem}}^2  
 \overset{\eqref{eq:norm.U.1T.le.norm.mu.1T}}\lesssim h_T\| \mu \|_{\partial \elem}^2. 
\end{equation}
The second inequality in \eqref{EQ:LS2} is derived by using \eqref{Proof:LS2 second} and summing over all elements.

Second, consider the first inequality on an arbitrary element $\elem \in \mesh_\level$.  Recalling the definition \eqref{EQ:hho_flux} of $\vec q_\elem(\cdot,\cdot)$ with $(u_\elem, m_{\partial \elem}) = (\localU_\elem \mu, \mu)$  and choosing $\vec p_\elem = \vec q_\elem(\localU_\elem \mu, \mu) \overset{\eqref{eq:localQ}}= \localQ_\elem \mu $,
$ u_\elem = \localU_\elem \mu$,
and $ m_{\partial \elem} =\mu$, we have
 \begin{equation*}
   \int_\elem \localQ_\elem \mu  \cdot \localQ_\elem \mu  \dx
   = \int_\elem \localU_\elem \mu\, (\Div \localQ_\elem \mu)\dx 
   - \int_{\partial \elem} \mu\, (\localQ_\elem \mu \cdot \Nu)  \ds.
  \end{equation*}
 Estimating the right-hand side of the above expression with Cauchy--Schwarz inequalities followed by
 the discrete inverse inequality \eqref{EQ:inverse_inequality} for the first term,
 the discrete trace inequality \eqref{EQ:discr_trace_inequality} for the second term, 
 we have 
 \begin{equation}\label{Proof:LS2 first}
   \| \localQ_\elem \mu  \|_\elem^2 \lesssim h^{-2}_\elem  \| \localU_\elem \mu \|_\elem^2 + h_\elem^{-1}  \| \mu \|_{\partial \elem}^2
   \overset{\eqref{Proof:LS2 second}}\leq  h_\elem^{-1}  \| \mu \|_{\partial \elem}^2.
 \end{equation}
Finally, the first inequality in \eqref{EQ:LS2} is derived by using \eqref{Proof:LS2 first}, summing over all elements, and using the mesh quasi-uniformity assumption to write
$
 \| \localQ_\level \mu \|_\Omega \lesssim  h^{-1}_\level  \|\mu\|_\level.
$

\subsection{Verification of \eqref{EQ:LS3}}

Plugging the definition \eqref{eq:localQ} of $\localQ_\elem$ into \eqref{EQ:hho_flux} we have,
\begin{equation*}
  \int_\elem \vec \localQ_\elem\mu \cdot \vec p_\elem \dx
  - \int_\elem \localU_\elem\mu\, (\Div \vec p_\elem) \dx 
     = - \int_{\partial \elem} \mu\, (\vec p_\elem \cdot \Nu) \ds
     \quad \forall \vec p_\elem \in \vec{W}_\elem.
\end{equation*}
Then, we integrate by parts the second term of the left-hand side and rearrange to infer
\begin{equation*}
  \int_\elem (\vec \localQ_\elem\mu + \nabla\localU_\elem\mu) \cdot \vec p_\elem \dx
  = \int_{\partial \elem} (\localU_\elem\mu - \mu) \vec p_\elem \cdot \Nu \ds
     \quad \forall \vec p_\elem \in \vec{W}_\elem.
\end{equation*}
Now, after noticing that $\nabla \localU_\elem \mu \in \nabla \mathcal{P}_p(\elem) \subset \nabla \mathcal{P}_{p+1}(\elem) = \vec W_T$, we can specify this relation for $\vec p_\elem = \localQ_\elem\mu + \nabla \localU_\elem\mu$, which gives
\begin{equation*}
  \|\localQ_\elem\mu + \nabla\localU_\elem\mu\|_T^2
  = \int_{\partial \elem} (\localU_\elem\mu - \mu)\, (\localQ_\elem\mu + \nabla\localU_\elem\mu) \cdot \Nu \ds.
\end{equation*}
Using the Cauchy--Schwarz inequality on the right-hand side, we have
% \corr{}{}{[DDP: I commented out the intermediate step in the following formula. As a matter of fact, if we want to be precise in removing the norma, we must say that we use a $(2,2,\infty)$-H\"{o}lder inequality along with the fact that $\| \Nu \|_{L^\infty(F)^d} \le 1$. I'll let you guy choose if you want to switch to this justification]}
\begin{align*}
    \|\localQ_\elem\mu + \nabla\localU_\elem\mu\|_T 
%%     &\leq 
%%     \|\localU_\elem\mu - \mu\|_{\partial T} \|[\localQ_\elem\mu + \nabla\localU_\elem\mu] \cdot \Nu\|_{\partial T} \\
    %%&
    \leq \|\localU_\elem\mu - \mu\|_{\partial T} \|\localQ_\elem\mu + \nabla\localU_\elem\mu\|_{\partial T}.
\end{align*}
At this point, we refer to the discrete trace inequality \eqref{EQ:discr_trace_inequality}, which we use component by component to bound the last term and obtain
\begin{equation*}
    \|\localQ_\elem\mu + \nabla\localU_\elem\mu\|_T 
    \lesssim 
    h_T^{-1/2}\|\localU_\elem\mu - \mu\|_{\partial T} \|\localQ_\elem\mu + \nabla\localU_\elem\mu\|_T.
\end{equation*}
Simplifying and squaring, we get
\begin{equation*}
    \|\localQ_\elem\mu + \nabla\localU_\elem\mu\|_T^2
    \lesssim 
    h_T^{-1}\|\localU_\elem\mu - \mu\|_{\partial T}^2.
\end{equation*}
Summing over all elements and using mesh quasi-uniformity yields
\[
\begin{aligned}
  \|\localQ_\elem\mu + \nabla\localU_\elem\mu\|_\Omega^2
  &\lesssim \sum_{T\in\mesh_\level }h_T^{-1}\|\localU_\elem\mu - \mu\|_{\partial T}^2
  \lesssim \sum_{T\in\mesh_\level }h_T^{-2} \sum_{F \in \faceSet_\elem} h_F \|\localU_\elem\mu - \mu\|_{\partial T}^2
  \\
  &\lesssim h_\level^{-2} \sum_{T\in\mesh_\level } \sum_{F \in \faceSet_\elem} h_F \|\localU_\elem\mu - \mu\|_{\partial T}^2
  \lesssim h_\level^{-1} \| \localU_\elem\mu - \mu \|_\level^2,
\end{aligned}
\]
where we have used the fact that $1 \lesssim \frac{h_\face}{h_\elem}$ for all $\elem \in \mesh_\level$ and all $\face \in \faceSet_\elem$ by mesh regularity in the second inequality,
the mesh quasi-uniformity assumption to write $h_T^{-2} \lesssim h_\level^{-2}$ in the third inequality, and \eqref{EQ:norm_l_rewrite} to conclude

\subsection{Verification of \eqref{EQ:LS4}}\label{SEC:ver_ls4}
We show \eqref{EQ:LS4} for a generic element $T \in \mesh_\ell$. If the identities hold there, they will hold on all elements. We need that $\gamma_\level \linElementSpace_\level$ is a subspace of $\skeletalSpace_\level$ and $V_\level$.
That is, we need that $p \ge 1$. We must show that HHO reproduces the $w \in \linElementSpace_{\level}$.

We take $w_{\partial T} \coloneq \gamma_\level w$.
\begin{equation*}
      \int_\elem \vec q_\elem(w, w_{\partial \elem}) \cdot \vec p_\elem \dx
      = 
      \int_\elem w \Div \vec p_\elem \dx
     - \int_{\partial \elem} w_{\partial \elem} \vec p_\elem \cdot \Nu \ds
     \quad \forall \vec p_\elem \in \vec{W}_\elem.
\end{equation*}
Integrating by parts the first term on the right-hand side, it holds that
\begin{align*}
    \int_\elem w \Div \vec p_\elem \dx
    = - \int_\elem \nabla w \cdot p_\elem\dx + \int_{\partial \elem} w_{\partial \elem} \vec p_\elem \cdot \Nu \ds
\end{align*}
Plugging that yields
\begin{equation*}
    \int_\elem \vec q_\elem(w, w_{\partial \elem}) \cdot \vec p_\elem \dx
    =
    - \int_\elem \nabla w \cdot p_\elem\dx
    \quad \forall \vec p_\elem \in \vec{W}_\elem,
\end{equation*}
which shows that 
\begin{equation} \label{EQ:tmp0}
    \vec q_\elem(w, w_{\partial \elem}) = - \nabla w.
\end{equation}

Given the definition \eqref{EQ:hho_localU} of $\localU_\elem$, we need to show that 
\begin{equation*}
\underline{a}_T((w, w_{\partial \elem}), (v_\elem, 0)) = 0 \qquad \forall v_\elem \in V_\elem.
\end{equation*}
We have
\begin{align*}
   \underline{a}_\elem(w, w_{\partial \elem}, (v_\elem, 0)) 
   &= \int_\elem \vec q_\elem(w, w_{\partial \elem}) \cdot \vec q_\elem(v_\elem, 0) \dx + \underbrace{\underline{s}_\elem ((w, w_{\partial \elem}), (v_\elem, 0))}_{0} \\
   &=
   -\int_\elem \nabla w \cdot \vec q_\elem(v_\elem, 0) \dx
  \end{align*}
where the stabilization term vanishes due to the polynomial consistency of $\underline{s}_\elem$ \cite[Assumption 2.4 (S3)]{di_pietro_hybrid_2020}, and we have used \eqref{EQ:tmp0} for the last equality.
For the remaining term, \eqref{EQ:hho_flux} gives
\begin{equation*}
      \int_\elem \vec q_\elem(v_\elem, 0) \cdot \vec p_\elem \dx - \int_\elem v_\elem \Div \vec p_\elem \dx  = 0
     \quad \forall \vec p_\elem \in \vec{W}_\elem,
  \end{equation*}
which we specialize to $\vec p_\elem = -\nabla w$ to infer
\begin{equation*}
    - \int_\elem \vec q_\elem(v_\elem, 0) \cdot \nabla w \dx + \int_\elem v_\elem \Delta w \dx = 0.
\end{equation*}
As $w\in \linElementSpace_\level$, $\Delta w = 0$, which concludes the proof of $\localU \gamma_\level w = w$.

\subsection{Verification of \eqref{EQ:LS5}}
This is an extension of Theorem 2.27 in \cite{HHObookCEP2022} by Sobolev interpolation. Following Lemma 2.9, the left-hand side of \eqref{EQ:LS5} is bounded by the consistency error estimates. The \eqref{EQ:LS5} is derived using  Theorem 2.10.

\subsection{Verification of \eqref{EQ:LS6}}
Combining some preliminary results, we have that
\begin{equation*}
    \|\mu\|_\level^2 
    \overset{\eqref{EQ:norm_equiv1}}{\lesssim} \|\mu\|_{1,\level}^2 
    \overset{\eqref{EQ:hho_stability}}{\lesssim} a_\level(\mu, \mu) 
    \overset{\eqref{EQ:hho_stability}}{\lesssim} \|\mu\|_{1,\level}^2.
\end{equation*}
Thus, we only need to prove that
\begin{equation} \label{EQ:norm_equiv2}
    \|\mu\|_{1,\level}  \lesssim h_\level^{-1}\|\mu\|_\level.
\end{equation}

Inserting $0 = \localQ_\level\mu - \localQ_\level\mu$ into the gradient term of \eqref{EQ:hho_norm} and using the triangle inequality gives
\begin{align}
    \|\mu\|_{1,\level}^2 
    &\lesssim \|\nabla_\level\localU_\level \mu + \localQ_\level\mu\|_\Omega^2 + \|\localQ_\level\mu\|_\Omega^2 + |\mu|_{1,\level} \nonumber \\
    \overset{\eqref{EQ:LS3}}&\lesssim h_\level^{-2} \| \localU_\level \mu - \mu \|_\level^2 + \|\localQ_\level\mu\|_\Omega^2 + |\mu|_{1,\level}. \label{EQ:LS6:proof1}
\end{align}
In order to make $\|\cdot\|_\level$ appear in $|\mu|_{1,\level}$, we observe that, by mesh quasi-uniformity,
\begin{align}
    |\mu|_{1,\level} 
    &\simeq h_\level^{-2} \sum_{\elem \in \mesh_\level} \sum_{\face\in \faceSet_\elem} h_\level \|\localU_\level \mu - \mu\|_F^2 %\nonumber \\
    \overset{\eqref{EQ:norm_l_rewrite}}\simeq h_\level^{-2} \|\localU_\level \mu - \mu\|_\level^2. \label{EQ:LS6:proof2}
\end{align}
Plugging \eqref{EQ:LS6:proof2} into \eqref{EQ:LS6:proof1} gives
\begin{align}
    \|\mu\|_{1,\level}^2 
    &\lesssim h_\level^{-2} \| \localU_\level \mu - \mu \|_\level^2 + \|\localQ_\level\mu\|_\Omega^2 \overset{\eqref{EQ:LS1}}{\lesssim} \| \mu \|^2_{a_\level}. \label{EQ:LS6:proof6}
\end{align}
Finally, \eqref{EQ:norm_equiv2} follows from \eqref{EQ:LS2}.

\subsection{Verification of \eqref{EQ:LS7}}
This equality corresponds to \cite[Assumption 2.4 (S3)]{di_pietro_hybrid_2020}.

\subsection{Verificaton of \eqref{EQ:LS9}}
This is relation \eqref{EQ:hho_stability}.

\section{Injection operators for HHO} \label{SEC:hho_injection_operators}
Consider two successive levels $\level$ (fine) and $\level-1$ (coarse). 
Given the coarse faces $\faceSet_{\level-1}$, the mesh nestedness allows us to decompose $\faceSet_\level$ as the disjoint union  $\widehat{\faceSet}_{\level} \cup \mathring{\faceSet}_{\level}$, where
\begin{equation*}
\widehat{\faceSet}_{\level} \coloneq \{ \face \in \faceSet_{\level} \mid \exists \face_{\level-1} \in \faceSet_{\level-1} \textrm{ s.t. } \face \subset \face_{\level-1} \}, 
\qquad
\mathring{\faceSet}_{\level} \coloneq \faceSet_{\level}\setminus \widehat{\faceSet}_{\level}.
\end{equation*}
In the following, we introduce three injection operators used in \cite{matalon_h-multigrid_2021}, denoted by $I_\level^i$, $i\in\{1, 2, 3\}$. 
The first one is defined as
\begin{equation} \label{EQ:hho_injection_1}
    (I_\level^1 \mu)_{|\face} \coloneq 
    \begin{cases}
        \mu_{|\face} &\text{ if } \face\in\widehat{\faceSet}_{\level}, \\
        (\localU_{\level-1} \mu)_{|\face} &\text{ otherwise.}
    \end{cases}
\end{equation}
This first operator exploits the nestedness of the meshes to straightforwardly transfer values from coarse faces to their embedded fine faces. Regarding the fine faces that are \emph{not} geometrically included in the coarse skeleton (i.e.\ $\mathring{\faceSet}_{\level}$), we make use of the local solver: $\localU_{\level-1}$ builds a bulk function in the coarse cells, whose traces provide admissible approximations on the fine faces.

Instead of the straight injection for the embedded faces, an alternative is to also use the bulk functions reconstructed from the local solver. As $\localU_{\level-1}$ yields a discontinuous polynomial, and the faces of $\widehat{\faceSet}_{\level}$ are located at the interface of two coarse cells, we propose to take the average of the respective traces, i.e.
\begin{equation} \label{EQ:hho_injection_2}
    (I_\level^2 \mu)_{|\face} \coloneq \avg{\localU_{\level-1} \mu}_\face.
\end{equation}
Remark that this formula also holds for $\face\in\mathring{\faceSet}_{\level}$, as the average trace of a continuous bulk function on both sides of a face reduces to its regular trace.
This injection operator comes with advantages. 
First, it allows the mesh nestedness condition to be relaxed, paving the way to generalized multigrid methods on non-nested meshes, as developed in \cite{di_pietro_towards_2021}.
Additionally, a more complex formula can handle difficulties occurring at cell interfaces. 
Typically, using an adequate weighted average formula can yield robust convergence in the presence of large jumps in the diffusion coefficient; see \cite{matalon_h-multigrid_2021}.

Finally, the third injection operator leverages the higher-order reconstruction operator, a salient feature of the HHO methods:
\begin{equation} \label{EQ:hho_injection_3}
    %(I_\level^3 \mu)_{|\face} \coloneq \avg{\discProj_{\level-1}\theta_{\level-1}^{p+1}(\localU_{\level-1} \mu, \mu)}_\face.
    (I_\level^3 \mu)_{|\face} \coloneq \pi_\level^p\avg{\theta_{\level-1}^{p+1}(\localU_{\level-1} \mu, \mu)}_\face,
\end{equation}
where $\pi_\level^p \colon L^2(\faceSet_\level) \to \skeletalSpace_\level$ denotes the $L^2$-orthogonal projector onto the skeletal polynomial space of degree $p$. $I_\level^3$ is based on the same principle of the averaged trace as $I_\level^2$, except that the higher-order reconstruction operator $\theta_{\level-1}^{p+1}$ enables the gain of one extra polynomial degree in the approximation of the coarse error. Then, after computing the average trace, the polynomial degree is lowered back to its original value by applying the $L^2$-orthogonal projector onto the lower-order space.

To fit these injection operators into the framework, we demonstrate
\begin{lemma}
Under the assumptions \eqref{EQ:LS2} and \eqref{EQ:LS4}, the injection operators $I_\level^i$, $i\in\{1, 2, 3\}$ verify \eqref{EQ:IA1}--\eqref{EQ:IA2}.
\end{lemma}

\begin{proof}
Regarding $I_\level^1$, we refer to \cite[Lem.\ 3.2]{LuRK22a}, as they investigate the same injection operator: $I_\level^1$ is their third operator.

Let us consider $I_\level^2$:
condition \eqref{EQ:IA2} follows directly from \eqref{EQ:LS4} recalling that we assume here $p \ge 1$ (cf. Remark~\ref{REM:lowest_order} below).
The discrete trace inequality \eqref{EQ:discr_trace_inequality} along with mesh regularity, i.e., $\frac{|T|}{|\partial T|} \| \mu \|_{\partial T}^2 \lesssim \| \mu \|_T^2$  for all  $T \in \mesh_\level$, gives
\begin{equation*}
 \| \localU_{\level - 1} \mu \|_\level \lesssim \| \localU_{\level - 1} \mu \|_\Omega \lesssim \| \mu \|_{\level-1},
\end{equation*}
where the second inequality is \eqref{EQ:LS2}.
\eqref{EQ:IA1} now follows from the mean value property $\| \avg{\localU_{\level-1} \mu}_\face \|_\level \le 2 \| \localU_{\level - 1} \mu \|_\level$.

Let us consider \eqref{EQ:IA2} for $I_\level^3$. On one hand, \eqref{EQ:LS4} says that
\begin{equation*}
 - \localQ_{\level-1} \gamma_{\level-1} w = \nabla_\level \theta_{\level-1}^{p+1}(\localU_{\level-1} \gamma_{\level-1} w, \gamma_{\level-1} w) = \nabla w
\end{equation*}
for $w \in \linElementSpace$. This, in turn, implies that
\begin{equation*}
 \theta_{\level-1}^{p+1}(\localU_{\level-1} \gamma_{\level-1} w, \gamma_{\level-1} w) = w + c
\end{equation*}
for some constant $c \in \IR$.  On the other hand, the high-order reconstruction guarantees (cf.\ \eqref{EQ:hho_closure}) that the mean value of $\theta_{\level-1}^{p+1}(\localU_{\level-1} \gamma_{\level-1} w, \gamma_{\level-1} w)$ equals the mean values of $\localU_{\level-1} \gamma_{\level-1} w = w$ (through \eqref{EQ:LS4}), which implies that $c = 0$.

Next, we prove \eqref{EQ:IA1} for $I_\level^3$: Let us consider a coarse element $\elem$. We have that
\begin{align*}
  \sum_{\elem \in \mesh_{\level-1}} \sum_{\bar \elem \in \mesh_\level}^{\bar\elem \subset \elem} & \sqrt\frac{|\bar \elem|}{|\partial \bar \elem|} \| \pi_\level^p \theta_{\level-1}^{p+1}(\localU_{\level-1} \mu|_\elem, \mu|_{\partial \bar \elem}) \|_{\partial \elem}
  \lesssim \sum_{\elem \in \mesh_{\level-1}} \| \theta_{\level-1}^{p+1}(\localU_{\level-1} \mu|_\elem, \mu|_{\partial \elem}) \|_\elem \\
  & \le \sum_{\elem \in \mesh_{\level-1}} \| \theta_{\level-1}^{p+1}(\localU_{\level-1} \mu|_\elem, \mu|_{\partial \elem}) - \localU_{\level-1} \mu|_\elem\|_\elem + \| \localU_{\level-1} \mu|_\elem \|_\Omega \\
  &\lesssim
      \sum_{\elem \in \mesh_{\level-1}} h_\elem \| \underbrace{ \nabla \theta_{\level-1}^{p+1}(\localU_{\level-1} \mu|_\elem, \mu|_{\partial \elem}) }_{ = - \localQ \mu } -  \nabla \localU_{\level-1} \mu|_\elem \|_\elem
      + \| \localU_{\level-1} \mu|_\elem \|_\Omega
      \\
  &\lesssim
      h_\elem \| \mu \|_{a_{\level -1}}
      + \| \localU_{\level-1} \mu|_\elem \|_\Omega
    \lesssim
      \| \mu \|_{\level - 1}.
\end{align*}
Here, the first inequality uses the boundedness of the $L^2$-orthogonal projector along with discrete trace inequalities. The second passage follows inserting $\pm \localU_{\level-1} \mu$ into the norm and using a triangle inequality. The third inequality is obtained using a local Poincaré--Wirtinger inequality. The fourth passage is obtained using \eqref{EQ:LS3} and \eqref{EQ:LS1}. The last inequality uses \eqref{EQ:LS6} and \eqref{EQ:LS2}.
\end{proof}

\section{Numerical experiments} \label{SEC:numerical_experiments}

\subsection{Experimental setup}

The numerical tests reported in this section are performed on the Poisson problem \eqref{EQ:poisson_pb}, on two- and three-dimensional domains. Namely, the fully elliptic problem shall be tested on the unit square and cube, and the low-regularity problem shall be tested on the L-shape domain. The problem is discretized using the standard HHO method as described in Section \ref{SEC:disc_skel_methods}, with the classical stabilizing term \cite[Ex.\ 2.7 (2.22)]{di_pietro_hybrid_2020}. Tests will span the polynomial degrees $p=1,2,3$, leaving the special case of $p=0$ out of the scope, as it does not verify the hypotheses of our framework; see Remark \ref{REM:lowest_order}. Modal polynomial bases are used in cells and on faces to assemble the method (specifically, we use $L^2$-orthogonal Legendre bases).

\begin{remark} \label{REM:lowest_order}
 (The lowest order case) The HHO method with $p=0$ does not verify the error estimates \eqref{EQ:hho_cell_ukn_error} and \eqref{EQ:hho_face_ukn_error}. Consequently, \eqref{EQ:LS5} does not hold and the multigrid method might not be uniformly convergent. It is worth noting that this limitation is not just theoretical. It has been observed in practice in \cite{matalon_h-multigrid_2021,di_pietro_towards_2021}, where suboptimal convergence is reported for $p=0$ while the method exhibits optimal convergence for the higher orders, with no variation in any other parameter.
\end{remark}

The multigrid method is constructed as described in Section \ref{SEC:multigrid_algorithm}. Section \ref{SEC:hho_injection_operators} describes the injection operators we consider. The V-cycle is carried out using pointwise Gauss--Seidel smoothing iterations, arranged so that both the smoothing step and the multigrid iteration remain symmetric. Namely, for V(1,1), the pre-smoothing iteration is performed in the forward order and the post-smoothing iteration in the backward order.
For V(2,2), the pre- and post-smoothing procedures involve a single forward iteration followed by a single backward iteration. The grid hierarchy is built by successive refinements of an initial simplicial mesh, and the condensed problem is assembled at every level. The coarsest system is solved with the Cholesky factorization. The stopping criterion relies on the backward error $\|\mathbf{r}\|_2/\|\mathbf{b}\|_2$, where $\mathbf{r}$ denotes the residual of the algebraic system, $\mathbf{b}$ the right-hand side, and $\|\cdot\|_2$ the standard Euclidean norm applied to the vector space of coordinates. In all experiments, convergence is considered to be reached once $\|\mathbf{r}\|_2/\|\mathbf{b}\|_2 < 10^{-6}$ is satisfied.

\subsection{Numerical tests}

\subsubsection{Full regularity: the unit square}

The Poisson problem is set up for the manufactured solution $u\colon (x, y) \mapsto \sin(4\pi x)\sin(4\pi y)$. The unit square is discretized by a hierarchy of 7 structured triangular meshes. For each problem solved on this hierarchy, the number of face unknowns is given in Table \ref{tab:square_pb_sizes}. The largest problem size we consider involves three million unknowns. Considering the injection operator $I_\level^1$, we first state that the V(1,1) cycle diverges regardless of the value of $p$. One must raise the number of smoothing steps to achieve convergence. With the V(2,2) cycle, the method converges for $p\in\{1,2\}$ and still diverges for $p=3$. The numbers of iterations are reported in Table \ref{tab:square_I1_V22} (the symbol $\infty$ indicating divergence or a number of iterations $> 100$). Their mild increase as the number of levels grows indicates asymptotic optimality. Table \ref{tab:square_I2} now considers the injection operator $I_\level^2$. In V(1,1), contrary to the same cycle with $I_\level^1$, the method now converges for $p\in\{1,2\}$, indicating better robustness of $I_\level^2$ compared to $I_\level^1$. The number of iterations of the V(2,2) cycle illustrates the uniform convergence of the method for all values of $p$. Finally, using the injection operator $I_\level^3$, we see from the results of Table \ref{tab:square_I3} that the method converges uniformly with both cycles and for all values of $p$, making this injection operator the most robust amongst those considered here.

\begin{table}
    \centering
    \begin{tabular}{cccccc}
         \toprule
         Levels & 3 & 4 & 5 & 6 & 7 \\
         \midrule
         $p=1$ &  \num{6016} & \num{24320} &  \num{97792} & \num{392192} & \num{1570816} \\
         % \midrule
         $p=2$ &  \num{9024} & \num{36480} & \num{146688} & \num{588288} & \num{2356224} \\
         % \midrule
         $p=3$ & \num{12032} & \num{48640} & \num{195584} & \num{784384} & \num{3141632} \\
         \bottomrule
    \end{tabular}
    \caption{For the square domain, number of face unknowns at each level and for each value of $p$.}
    \label{tab:square_pb_sizes}
\end{table}

\begin{table}
    \centering
    \begin{tabular}{cccccc}
         \toprule
         Levels & 3 & 4 & 5 & 6 & 7 \\ \midrule
         $p=1$ & 13 & 14 & 14 & 15 & 15 \\ 
         $p=2$ & 36 & 38 & 38 & 39 & 39 \\ 
         $p=3$ & $\infty$ & $\infty$ & $\infty$ & $\infty$ & $\infty$ \\
         \bottomrule
    \end{tabular}
    \caption{In the square domain, number of V(2,2) iterations with the injection operator $I_\level^1$.}
    \label{tab:square_I1_V22}
\end{table}

% \begin{table}
%     \centering
%     \begin{tabular}{cccccc}
%          \toprule
%          Levels & 3 & 4 & 5 & 6 & 7 \\ \midrule
%          $p=1$ & 24 & 25 & 25 & 26 & 26 \\
%          $p=2$ & 20 & 22 & 25 & 26 & 27 \\
%          $p=3$ & $\infty$ & $\infty$ & $\infty$ & $\infty$ & $\infty$ \\
%          \bottomrule
%     \end{tabular}
%     \caption{Square domain, injection operator $I_\level^2$, V(1,1) cycle.}
%     \label{tab:square_I2_V11}
% \end{table}

% \begin{table}
%     \centering
%     \begin{tabular}{cccccc}
%          \toprule
%          Levels & 3 & 4 & 5 & 6 & 7 \\ \midrule
%          $p=1$ & 13 & 13 & 14 & 14 & 14 \\
%          $p=2$ & 10 & 10 & 11 & 11 & 11 \\
%          $p=3$ & 13 & 13 & 13 & 14 & 14 \\
%          \bottomrule
%     \end{tabular}
%     \caption{Square domain, injection operator $I_\level^2$, V(2,2) cycle.}
%     \label{tab:square_I2_V22}
% \end{table}

\begin{table}
    \centering
    \begin{tabular}{c|ccccc|ccccc}
         \toprule
               & \multicolumn{5}{c|}{V(1,1)} & \multicolumn{5}{c}{V(2,2)} \\
         Levels & 3 &  4 &  5 &  6 &  7 &  3 &  4 &  5 &  6 &  7 \\ \midrule
         $p=1$ & 24 & 25 & 25 & 26 & 26 & 13 & 13 & 14 & 14 & 14 \\
         $p=2$ & 20 & 22 & 25 & 26 & 27 & 10 & 10 & 11 & 11 & 11 \\
         $p=3$ & $\infty$ & $\infty$ & $\infty$ & $\infty$ & $\infty$ & 13 & 13 & 13 & 14 & 14 \\
         \bottomrule
    \end{tabular}
    \caption{In the square domain, number of iterations with the injection operator $I_\level^2$.}
    \label{tab:square_I2}
\end{table}

% \begin{table}
%     \centering
%     \begin{tabular}{cccccc}
%          \toprule
%          Levels & 3 & 4 & 5 & 6 & 7 \\ \midrule
%          $p=1$ & 18 & 18 & 19 & 19 & 20 \\
%          $p=2$ & 17 & 17 & 17 & 17 & 18 \\
%          $p=3$ & 20 & 21 & 21 & 21 & 21 \\
%          \bottomrule
%     \end{tabular}
%     \caption{Square domain, injection operator $I_\level^3$, V(1,1) cycle.}
%     \label{tab:square_I3_V11}
% \end{table}

% \begin{table}
%     \centering
%     \begin{tabular}{cccccc}
%          \toprule
%          Levels & 3 & 4 & 5 & 6 & 7 \\ \midrule
%          $p=1$ & 10 & 10 & 11 & 11 & 11 \\ %\midrule
%          $p=2$ &  9 & 10 & 10 & 10 & 10 \\ %\midrule
%          $p=3$ & 11 & 11 & 11 & 11 & 11 \\
%          \bottomrule
%     \end{tabular}
%     \caption{Square domain, injection operator $I_\level^3$, V(2,2) cycle.}
%     \label{tab:square_I3_V22}
% \end{table}

\begin{table}
    \centering
    \begin{tabular}{c|ccccc|ccccc}
         \toprule
               & \multicolumn{5}{c|}{V(1,1)} & \multicolumn{5}{c}{V(2,2)} \\
         Levels & 3 &  4 &  5 &  6 &  7 &  3 &  4 &  5 &  6 &  7 \\ \midrule
         $p=1$ & 18 & 18 & 19 & 19 & 20 & 10 & 10 & 11 & 11 & 11 \\
         $p=2$ & 17 & 17 & 17 & 17 & 18 &  9 & 10 & 10 & 10 & 10 \\
         $p=3$ & 20 & 21 & 21 & 21 & 21 & 11 & 11 & 11 & 11 & 11 \\
         \bottomrule
    \end{tabular}
    \caption{In the square domain, number of iterations with the injection operator $I_\level^3$.}
    \label{tab:square_I3}
\end{table}

\subsubsection{Low regularity: the L-shape domain}
The computational domain is now the L-shape domain $\Omega = (-1, 1)^2 \setminus ([0, 1]\times[0, -1])$. The source function of the Poisson problem is set to zero. The Dirichlet boundary condition follows the manufactured exact solution $u\colon(r,\varphi) \mapsto r^{2/3}\sin(\frac{2}{3}\varphi)$, where $(r,\varphi)$ represent polar coordinates of domain points. An unstructured Delaunay triangulation builds the coarsest mesh, and a classical refinement technique by edge bisection is employed to build finer meshes successively. Although the theory requires the variable cycle to conclude, we notice that fixed cycles exhibit uniform convergence in practice. Therefore, the test results presented here are obtained with classical V(1,1) and V(2,2) cycles. Similarly to the preceding test case, the problem sizes are indicated in Table \ref{tab:L_shape_pb_sizes}. The results, presented in Tables \ref{tab:L_shape_I1_V22}, \ref{tab:L_shape_I2}, and \ref{tab:L_shape_I3} are also qualitatively similar. With the least robust injection operator, $I_\level^1$, the V(1,1) cycle diverges for all values of $p$, and the V(2,2) cycle is uniformly convergent only for $p=1$ (cf.\ Table \ref{tab:L_shape_I1_V22}). With $I_\level^2$, V(1,1) provides a uniformly convergent solver up to $p=2$, while V(2,2) converges uniformly for all $p\in\{1,2,3\}$ (cf.\ Table \ref{tab:L_shape_I2}). Finally, using $I_\level^3$, Table \ref{tab:L_shape_I3} shows that the convergence is optimal for both cycles and all values of $p$. Based on these experiments, one can conclude that the convergence of the multigrid solver does not seem to suffer from the lower regularity of the solution. However, if we compare the number of iterations obtained with $I_\level^2$ to those obtained with $I_\level^3$, we remark that robustness aside, $I_\level^3$ does not generally yield a faster convergence (see especially V(2,2)). Recall that $I_\level^3$ is obtained from $I_\level^2$ by inserting the higher-order reconstruction operator. Since it is known that sufficient regularity is required for higher orders to have beneficial effects, the low regularity of the solution might explain this lack of noticeable improvement. Looking back at the fully regular problem on the square, one sees that the higher-order reconstruction did improve convergence in that case (Table \ref{tab:square_I2} vs.\ Table \ref{tab:square_I3}).

\begin{table}
    \centering
    \begin{tabular}{cccccc}
         \toprule
         Levels & 3 & 4 & 5 & 6 & 7 \\
         \midrule
             $p=1$ & \num{4480} & \num{18176} &  \num{73216} & \num{293888} & \num{1177600} \\
             $p=2$ & \num{6720} & \num{27264} & \num{109824} & \num{440832} & \num{1766400} \\
             $p=3$ & \num{8960} & \num{36352} & \num{146432} & \num{587776} & \num{2355200} \\
         \bottomrule
    \end{tabular}
    \caption{For the L-shape domain, number of face unknowns at each level and for each value of $p$.}
    \label{tab:L_shape_pb_sizes}
\end{table}

\begin{table}
    \centering
    \begin{tabular}{cccccc}
         \toprule
         Levels & 3 &  4 &  5 &  6 &  7 \\ \midrule
         $p=1$ & 15 & 12 & 14 & 15 & 15 \\ 
         $p=2$ & $\infty$ & $\infty$ & $\infty$ & $\infty$ & $\infty$ \\ 
         $p=3$ & $\infty$ & $\infty$ & $\infty$ & $\infty$ & $\infty$ \\
         \bottomrule
    \end{tabular}
    \caption{In the L-shape domain, number of V(2,2) iterations with the injection operator $I_\level^1$.}
    \label{tab:L_shape_I1_V22}
\end{table}

\begin{table}
    \centering
    \begin{tabular}{c|ccccc|ccccc}
         \toprule
               & \multicolumn{5}{c|}{V(1,1)} & \multicolumn{5}{c}{V(2,2)} \\
         Levels & 3 &  4 &  5 &  6 &  7 &  3 &  4 &  5 &  6 &  7 \\ \midrule
         $p=1$ & 20 & 20 & 20 & 20 & 20 & 11 & 11 & 11 & 11 & 11 \\
         $p=2$ & 17 & 17 & 17 & 17 & 17 &  9 &  9 &  9 &  9 &  9 \\
         $p=3$ & $\infty$ & $\infty$ & $\infty$ & $\infty$ & $\infty$ & 11 & 11 & 11 & 11 & 11 \\
         \bottomrule
    \end{tabular}
    \caption{In the L-shape domain, number of iterations with the injection operator $I_\level^2$.}
    \label{tab:L_shape_I2}
\end{table}

\begin{table}
    \centering
    \begin{tabular}{c|ccccc|ccccc}
         \toprule
               & \multicolumn{5}{c|}{V(1,1)} & \multicolumn{5}{c}{V(2,2)} \\
         Levels & 3 &  4 &  5 &  6 &  7 &  3 &  4 &  5 &  6 &  7 \\ \midrule
         $p=1$ & 16 & 16 & 16 & 16 & 16 &  9 &  9 & 10 & 11 & 11 \\
         $p=2$ & 17 & 17 & 17 & 17 & 17 &  8 &  9 &  9 &  9 &  9 \\
         $p=3$ & 20 & 20 & 20 & 20 & 20 & 10 & 10 & 10 & 10 & 10 \\
         \bottomrule
    \end{tabular}
    \caption{In the L-shape domain, number of iterations with the injection operator $I_\level^3$.}
    \label{tab:L_shape_I3}
\end{table}

\subsubsection{3D test case: the cubic domain}

The unit cube is discretized by a Cartesian grid, where each element is decomposed into six geometrically similar tetrahedra, following \cite[Fig.\ 9]{bey_tetrahedral_1995}. 
Repeating the same procedure on successive levels of embedded Cartesian grids ensures that the subsequent tetrahedral grids are also embedded.
Five levels are built, and the considered problem sizes are given in Table \ref{tab:cube_pb_sizes}.
The manufactured solution is $u\colon (x, y) \mapsto \sin(4\pi x)\sin(4\pi y)\sin(4\pi z)$.
This problem is more challenging, so we restrict ourselves to the most efficient injection operator $I_\level^3$ and the V(2,2) cycle (V(1,1) diverges). 
The convergence results presented in Table \ref{tab:cube_I3_V22}, while still exhibiting asymptotically optimal convergence, also show higher numbers of iterations than in 2D.
More numerical tests are available in \cite{matalon_h-multigrid_2021}, where it is shown that better convergence can be achieved for comparable cost through simple parameter tunings such as the use of cycles with post-smoothing only, non-alternating directions in the Gauss-Seidel sweeps, or blockwise Gauss-Seidel smoothers.

\begin{table}
    \centering
    \begin{tabular}{cccccc}
         \toprule
         Levels & 3 & 4 & 5 \\
         \midrule
             $p=1$ & \num{17280} & \num{142848} & \num{1161216} \\
             $p=2$ & \num{34560} & \num{285696} & \num{2322432} \\
             $p=3$ & \num{57600} & \num{476160} & \num{3870720} \\
         \bottomrule
    \end{tabular}
    \caption{For the cubic domain, number of face unknowns at each level and for each value of $p$.}
    \label{tab:cube_pb_sizes}
\end{table}

\begin{table}
    \centering
    \begin{tabular}{cccccc}
         \toprule
         Levels &  3 &  4 &  5 \\ \midrule
         $p=1$  & 18 & 23 & 25 \\ 
         $p=2$  & 23 & 24 & 22 \\ 
         $p=3$  & 22 & 23 & 23 \\
         \bottomrule
    \end{tabular}
    \caption{In the cubic domain, number of V(2,2) iterations with the injection operator $I_\level^3$.}
    \label{tab:cube_I3_V22}
\end{table}

\subsection{Conclusions}
We have constructed and rigorously analyzed a homogeneous multigrid method for HHO methods, which does not internally change the discretization scheme. To this end, we generalized the homogeneous multigrid framework for HDG methods and verified the resulting, more general assumptions for HHO methods.  We verified our analytical findings numerically using several injection operators.

\section*{Acknowledgements}

The work of D.\ Di Pietro was partially funded by the European Union (ERC Synergy, NEMESIS, project number 101115663).
Views and opinions expressed are however those of the authors only and do not necessarily reflect those of the European Union or the European Research Council Executive Agency. Neither the European Union nor the granting authority can be held responsible for them.

G.\ Kanschat has been supported by the Deutsche Forschungsgemeinschaft (DFG, German Research Foundation) under Germany's Excellence Strategy EXC 2181/1 - 390900948 (the Heidelberg STRUCTURES Excellence Cluster).

A.\ Rupp has been supported by the Academy of Finland's grant number 350101 \emph{Mathematical models and numerical methods for water management in soils}, grant number 354489 \emph{Uncertainty quantification for PDEs on hypergraphs}, grant number 359633 \emph{Localized orthogonal decomposition for high-order, hybrid finite elements}, Business Finland's project number 539/31/2023 \emph{3D-Cure: 3D printing for personalized medicine and customized drug delivery}, and the Finnish \emph{Flagship of advanced mathematics for sensing, imaging and modelling}, decision number 358944.

The authors thank Ulrich R\"ude (FAU Erlangen-Nuremberg) for his support.

%\clearpage
\bibliographystyle{ARalpha}
\bibliography{hmg_discrete_skeletal}
\end{document}